\documentclass[11pt]{article}
\usepackage[dvipsnames]{xcolor}
\usepackage{amssymb,amsthm,amsmath,amsfonts,latexsym,tikz,hyperref,color,enumitem}
\usepackage[hmargin=1in,vmargin=1in]{geometry}
\usepackage[normalem]{ulem}

\setlist[enumerate]{itemsep=-1ex}
\setlist[itemize]{itemsep=-1ex}

\newcommand{\Up}{\Upsilon}

\newcommand{\excise}[1]{}
\definecolor{light}{gray}{.85}
\definecolor{kindoflight}{gray}{.75}
\renewcommand{\AA}{\mathcal{A}}
\newcommand{\Cc}{\mathbb{C}}
\newcommand{\Rr}{\mathbb{R}}
\newcommand{\join}{\vee}

\newtheorem{thm}{Theorem}[section]
\newtheorem{prop}[thm]{Proposition}
\newtheorem{cor}[thm]{Corollary}
\newtheorem{lem}[thm]{Lemma}
\newtheorem{conj}[thm]{Conjecture}
\newtheorem{question}[thm]{Question}
\theoremstyle{definition}
\newtheorem{exa}[thm]{Example}
\newtheorem{defn}[thm]{Definition}
\theoremstyle{remark}

\DeclareMathOperator{\brk}{br}
\DeclareMathOperator{\link}{link}

\DeclareMathOperator{\cISF}{\cI\cS\cF}
\DeclareMathOperator{\isf}{isf}
\DeclareMathOperator{\ISF}{ISF}
\DeclareMathOperator{\cCF}{\cC\cF}
\DeclareMathOperator{\cf}{cf}
\DeclareMathOperator{\CF}{CF}
\DeclareMathOperator{\cTF}{\cT\cF}
\DeclareMathOperator{\tf}{tf}
\DeclareMathOperator{\TF}{TF}
\DeclareMathOperator{\nbc}{nbc}

\DeclareMathOperator{\ao}{ao}

\newcommand{\ben}{\begin{enumerate}}
\newcommand{\een}{\end{enumerate}}
\newcommand{\ble}{\begin{lem}}
\newcommand{\ele}{\end{lem}}
\newcommand{\bth}{\begin{thm}}
\renewcommand{\eth}{\end{thm}}
\newcommand{\bpr}{\begin{prop}}
\newcommand{\epr}{\end{prop}}
\newcommand{\bco}{\begin{cor}}
\newcommand{\eco}{\end{cor}}
\newcommand{\bcon}{\begin{conj}}
\newcommand{\econ}{\end{conj}}
\newcommand{\bde}{\begin{defn}}
\newcommand{\ede}{\end{defn}}
\newcommand{\bex}{\begin{exa}}
\newcommand{\eex}{\end{exa}}
\newcommand{\barr}{\begin{array}}
\newcommand{\earr}{\end{array}}
\newcommand{\btab}{\begin{tabular}}
\newcommand{\etab}{\end{tabular}}
\newcommand{\beq}{\begin{equation}}
\newcommand{\eeq}{\end{equation}}
\newcommand{\bea}{\begin{eqnarray*}}
\newcommand{\eea}{\end{eqnarray*}}
\newcommand{\bal}{\begin{align*}}
\newcommand{\bce}{\begin{center}}
\newcommand{\ece}{\end{center}}
\newcommand{\bpi}{\begin{picture}}
\newcommand{\epi}{\end{picture}}
\newcommand{\bpp}{\begin{picture}}
\newcommand{\epp}{\end{picture}}
\newcommand{\bfi}{\begin{figure} \begin{center}}
\newcommand{\efi}{\end{center} \end{figure}}
\newcommand{\bprf}{\begin{proof}}
\newcommand{\eprf}{\end{proof}\medskip}
\newcommand{\capt}{\caption}
\newcommand{\bsl}{\begin{slide}{}}
\newcommand{\esl}{\end{slide}}
\newcommand{\bfr}{\begin{frame}}
\newcommand{\efr}{\end{frame}}

\newcommand{\hqed}{\hfill \qed}

\newcommand{\hso}[1]{\hspace{-1pt}}
\newcommand{\vs}[1]{\vspace{#1}}

\newcommand{\emp}{\emptyset}

\newcommand{\sm}{\hspace{-2pt}\setminus\hspace{-2pt}}

\newcommand{\sbe}{\subseteq}

\def\<{\langle}
\def\>{\rangle}

\newcommand{\ree}[1]{(\ref{#1})}

\newcommand{\si}{\sigma}

\newcommand{\De}{\Delta}

\newcommand{\0}{{\bf 0}}

\newcommand{\bx}{{\bf x}}

\newcommand{\bbZ}{{\mathbb Z}}

\newcommand{\cC}{{\cal C}}

\newcommand{\cF}{{\cal F}}

\newcommand{\cI}{{\cal I}}

\newcommand{\cNBC}{{\cal NBC}}

\newcommand{\cS}{{\cal S}}
\newcommand{\cT}{{\cal T}}

\DeclareMathOperator{\rk}{rk}

\begin{document}

\pagestyle{plain}

\title{Increasing spanning forests in graphs and simplicial complexes}

\author{%
Joshua  Hallam%
\thanks{Department of Mathematics and Statistics, Wake Forest University, Winston-Salem, NC 27109, USA.
E-mail: {\tt jwhalla@gmail.com}}
\and
Jeremy L.\ Martin%
\thanks{Department of Mathematics, University of Kansas, Lawrence, KS 66045-7594, USA.
E-mail: {\tt jlmartin@ku.edu}}
\and
Bruce E.\ Sagan%
\thanks{Department of Mathematics, Michigan State University, East Lansing, MI 48824-1027, USA.
E-mail: {\tt sagan@math.msu.edu}}
}

\date{\today\\[10pt]
	\begin{flushleft}
	                                       \\[5pt]
	\end{flushleft}}

\maketitle

\begin{abstract}

Let $G$ be a graph with vertex set $\{1,\dots,n\}$.  A spanning forest $F$ of $G$ is {\em increasing} if the sequence of labels on any path starting at the minimum vertex of a tree of $F$ forms an increasing sequence.  Hallam and Sagan showed that the  generating function $\ISF(G,t)$ for increasing spanning forests of $G$ has all nonpositive integral roots.  Furthermore they proved that, up to a change of sign, this polynomial equals the chromatic polynomial of $G$ precisely when $1,\dots,n$  is a perfect elimination order for $G$.  We give new, purely combinatorial proofs of these results which permit us to generalize them in several ways. For example, we are able to bound the coefficients of $\ISF(G,t)$ using broken circuits.  We are also able to extend these results to simplicial complexes using the new notion of a cage-free complex.  A generalization to labeled multigraphs is also given.  We observe that the definition of an increasing spanning forest can be formulated in terms of pattern
avoidance, and we end by exploring spanning forests that avoid the patterns $231$, $312$ and $321$.

\end{abstract}

%
%

\section{Introduction}

The purpose of this paper is to prove generalizations and consequences of two theorems of Hallam and Sagan about  increasing spanning forests~\cite{hs:fcp}.  To state them, we first need some definitions.

Let $G=(V,E)$ be a finite graph with vertex set $V=V(G)$ and edge set $E=E(G)$.  We will always assume that $V$ is a subset of the positive integers so that there is a total order on the vertices.  If the graph is a tree $T$ then we consider it to be rooted at its smallest vertex $r$.

\begin{defn} \label{defn:isf}
A labeled tree is \emph{increasing} if the integers on any path beginning at the root form an increasing sequence.
A labeled forest is \emph{increasing} if each component is an increasing tree.
\end{defn}

For example, the tree on the left in Figure~\ref{T} is increasing, while the one on the right is not because of the path $2,7,4$.

\begin{figure}[th]
\bce
\begin{tikzpicture}
\coordinate (v2) at (3,2);
\coordinate (v3) at (1.5,1);
\coordinate (v6) at (3,1);
\coordinate (v4) at (4.5,1);
\coordinate (v5) at (1,0);
\coordinate (v9) at (2,0);
\coordinate (v8) at (4,0);
\coordinate (v7) at (5,0);
\foreach \v in {(v2),(v3),(v4),(v5),(v6),(v7),(v8),(v9)} \fill \v circle(.1);
\draw(1,-.4) node{$5$};
\draw(2,-.4) node{$9$};
\draw(4,-.4) node{$8$};
\draw(5,-.4) node{$7$};
\draw(1.5,1.4) node{$3$};
\draw(3,.6) node{$6$};
\draw(4.5,1.4) node{$4$};
\draw(3,2.4) node{$2$};
\draw (v5) -- (v3) -- (v2) -- (v4) -- (v7) (v9) -- (v3) (v6) -- (v2) (v8) -- (v4);
\node at (3,-1) {$T$};
\begin{scope}[shift={(8,0)}]
\coordinate (v2) at (3,2);
\coordinate (v3) at (1.5,1);
\coordinate (v6) at (3,1);
\coordinate (v7) at (4.5,1);
\coordinate (v5) at (1,0);
\coordinate (v9) at (2,0);
\coordinate (v8) at (4,0);
\coordinate (v4) at (5,0);
\foreach \v in {(v2),(v3),(v4),(v5),(v6),(v7),(v8),(v9)} \fill \v circle(.1);
\draw(1,-.4) node{$5$};
\draw(2,-.4) node{$9$};
\draw(4,-.4) node{$8$};
\draw(5,-.4) node{$4$};
\draw(1.5,1.4) node{$3$};
\draw(3,.6) node{$6$};
\draw(4.5,1.4) node{$7$};
\draw(3,2.4) node{$2$};
\draw (v5) -- (v3) -- (v2) -- (v6) (v7) -- (v8) (v9) -- (v3);
\draw[thick,dashed] (v2) -- (v7) -- (v4);
\node at (3,-1) {$T'$};
\end{scope}
\end{tikzpicture}
\ece
\capt{An increasing tree $T$ and a non-increasing tree $T'$. \label{T}}
\end{figure}

As usual, call a subgraph $H$ of $G$ {\em spanning} if $V(H)=V(G)$.  Since a spanning subgraph is determined by its edge set, it is convenient to ignore the distinction between subsets of $E(G)$ and spanning subgraphs.
We will be interested in \emph{increasing spanning forests} of $G$, or ISFs for short.  Define
\begin{equation} \label{ISF-notation}
\begin{aligned}
\cISF(G) &=\text{ set of ISFs of $G$}, & \isf(G) &= |\cISF(G)|,\\
\cISF_m(G)&=\text{set of ISFs of $G$ with $m$ edges}, & \isf_m(G) &= |\cISF_m(G)|,\\
\ISF(G,t) &= \sum_{m\ge0} \isf_m(G) t^{|V(G)|-m}
\end{aligned}
\end{equation}
where the absolute value signs denote cardinality.
These invariants depend on the labeling of the vertices of $G$, although the notation does not specify the labeling explicitly.
By way of illustration, consider the two labelled graphs in Figure~\ref{G}.  An easy computation shows that
\[\begin{array}{lll}
\ISF(G,t) = t^4+4t^3+5t^2+2t &=& t(t+1)^2(t+2),\\
\ISF(H,t) = t^4+4t^3+3t^2 &=& t^2(t+1)(t+3).
\end{array}\]
Even though these polynomials are different, it is striking that they both factor with nonpositive integral roots.  

To explain the two previous factorizations, it will be convenient to assume from now on that for all graphs $G$ we have
$V(G)=\{1,2,\dots,n\}:=[n]$ unless otherwise noted.  We will also adopt the convention that all edges will be listed with their smallest vertex first.  For $k\in[n]$ we define
\beq
\label{E_k}
E_k= E_k(G) =\{ e\in E(G)\ :\ \text{$e = jk$ for some $ j<k$}\}.
\eeq
Returning to the left-hand graph in Figure~\ref{G}, we have
$$
E_1=\emp,\  E_2=\{12\},\ E_3=\{23\},\ E_4=\{14,24\}.
$$
Note that 
$$
\prod_{k=1}^4 (t+|E_k|) = t(t+1)^2(t+2) = \ISF(G,t).
$$
It turns out that this is always the case.
\bth[{\cite[Theorem~25]{hs:fcp}}]
\label{HS1}
Let $G$ be a graph with $V=[n]$ and $E_k$ as in~\ree{E_k}.  Then
\[\ISF(G,t) = \prod_{k=1}^n (t+|E_k|).\]
\eth

\bfi
\begin{tikzpicture}
\foreach \v in {(0,0),(2,0),(0,2),(2,2)} \fill \v circle(.1);
\draw(2,0)--(2,2)--(0,2)--(0,0)--(2,2);
\draw(-.5,0) node{$4$};
\draw(2.5,0) node{$3$};
\draw(-.5,2) node{$1$};
\draw(2.5,2) node{$2$};
\draw(1,-1) node{\Large $G$};
\begin{scope}[shift={(8,0)}]
\foreach \v in {(0,0),(2,0),(0,2),(2,2)} \fill \v circle(.1);
\draw(2,0)--(2,2)--(0,2)--(0,0)--(2,2);
\draw(-.5,0) node{$2$};
\draw(2.5,0) node{$3$};
\draw(-.5,2) node{$1$};
\draw(2.5,2) node{$4$};
\draw(1,-1) node{\Large $H$};
\end{scope}
\end{tikzpicture}
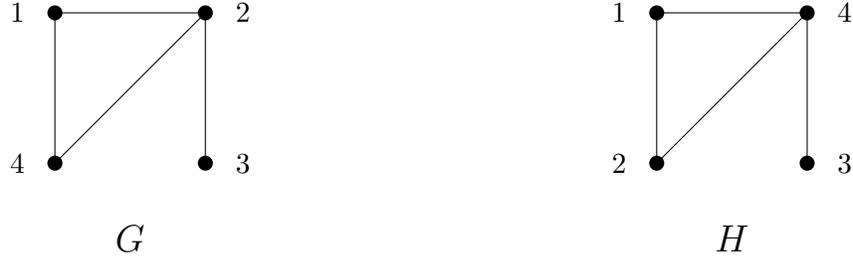
\capt{Two graphs $G$ and $H$  \label{G}}
\efi

To state the second theorem we will be studying, we recall some notions from the theory of graph coloring.   A {\em proper coloring} of a graph $G$ using a set $S$ is an assignment of elements of $S$ to the vertices of $G$ so that no edge has both endpoints the same color.  Suppose $t$ is a positive integer.  The {\em chromatic polynomial} of $G$  is
$$
P(G,t)=\text{the number of proper colorings of $V$ using the set $[t]$.}
$$
It is well known that $P(G,t)$ is a polynomial function of~$t$ for every graph~$G$.
Returning again to the graph~$G$ in Figure~\ref{G}, if we color the vertices in the order $1,2,3,4$, then the number of colors available at each step depends only on the number of adjacent, previously-colored vertices.  So we obtain
$$
P(G,t)=t(t-1)(t-1)(t-2)=(-1)^4 \ISF(G,-t).
$$ 

It cannot be the case that $P(G,t)$ and $\ISF(G,t)$ are always the same up to a sign change: the chromatic polynomial does not always have integral roots, and is independent of the choice of  labeling.  However, there is a well-known condition which implies equality.

\begin{defn}\label{def:PEO}
A \emph{perfect elimination ordering (PEO)} on $G$ is a total ordering  $v_1,v_2,\dots,v_n$ of $V(G)$ such that, for every $k$, the set $N(v_k,G_k):=N(v_k)\cap\{v_1,\dots,v_{k-1}\}$ is a clique, where $N(v_k)$ denotes the set of neighbors of $v_k$.

Equivalently, for all $i<j<k$, if $v_iv_k$ and $v_jv_k\in E(G)$, then $v_iv_j\in E(G)$.
\end{defn}
It is well known that the existence of a PEO is equivalent to the condition that $G$ is \emph{chordal}, i.e., every cycle of length~4 or greater has a chord (an edge between two non-consecutive vertices of the cycle).  

If $G$ has a PEO, then counting exactly as we did for our example graph gives
$$
P(G,t)=\prod_{k=1}^n (t-|N(v_k,G_k)|).
$$
It is also easy to verify that the order $1,2,3,4$ is a PEO for the graph $G$ in Figure~\ref{G}, while that same order is not a PEO for the graph $H$.  Again, this presages a general result.
\bth[\cite{hs:fcp}]
\label{HS2}
Let $G$ be a graph with $V=[n]$.  We have
$$
\ISF(G,t)=(-1)^n P(G,-t)
$$
if and only if the ordering $1,2,\dots,n$ is a PEO for $G$.
\eth

This paper is devoted to expanding the ideas of~\cite{hs:fcp} to broader settings, including replacing graphs with simplicial complexes or labeled multigraphs, or replacing ISFs with labeled forests obeying more general pattern-avoidance conditions.

Section~\ref{otr} contains new, combinatorial proofs of strengthened versions of Theorems~\ref{HS1} and~\ref{HS2}.  The previous proofs used the machinery of poset quotients developed in~\cite{hs:fcp}.
 First, we give an alternative characterization of increasing forests  in Lemma~\ref{ik,jk}, that implies a weighted factorization formula, Theorem~\ref{ISF-weighted-factorization}, which generalizes Theorem~\ref{HS1}.  Second, we show in Theorem~\ref{isf-and-nbc} that there is a bijection between increasing spanning forests in $G$ and NBC~sets (edge sets with no broken circuit) precisely when the natural ordering on $G$ is a PEO.  Theorem~\ref{HS2} follows from this result together with Whitney's classical interpretation of the chromatic polynomial as a generating function for NBC~sets \cite{whi:lem}.

In Section~\ref{sc}, we extend our results from graphs to simplicial complexes.  The new characterization of increasing spanning forests in Lemma~\ref{ik,jk} naturally generalizes to the idea of a \emph{cage-free} subcomplex of a simplicial complex in Definition~\ref{caged}, and so we study the generating function
\[
\CF(\De,t,\bx)=\sum_\Upsilon \prod_{\phi\in\Upsilon_d} x_\phi t^{N-|\Up_d|}
\]
for cage-free subcomplexes $\Upsilon$ of a pure 
simplicial complex $\Delta$ of dimension $d$, where $\Upsilon_d$ denotes the set of $d$-faces of $\Upsilon$, the $x_\phi$ are indeterminates, and $N$ is a certain integer.  This generating function admits a factorization given in Theorem~\ref{CFfactor} that  generalizes Theorem~\ref{ISF-weighted-factorization}.  Moreover, the specialization obtained by setting $x_\phi=1$ is essentially the product of generating functions for increasing spanning forests of graphs $G_\sigma$ corresponding to certain codimension-2 faces of $\Delta$ as shown in Proposition~\ref{CFdeltaEqCFgraph}).  We conclude the section with a discussion of the difficulty of extending the definition of a perfect elimination order (PEO) to higher dimension.

Section~\ref{multi} generalizes the theory to \emph{labeled multigraphs}: graphs with multiple edges permitted, labeled by nonzero complex numbers.  The definition of an increasing spanning forest and the factorization formula for the ISF generating function in Theorem~\ref{HS1}  carry over easily to this setting.  In this setting, the definition of a perfect elimination order is somewhat more subtle, since it relies on both the vertex ordering and the complex edge labeling.  We associate a complex hyperplane arrangement $\AA(G)$ to each labeled multigraph $G$, generalizing the standard construction of a graphic arrangement, and prove in Theorem~\ref{ISFandChi} that the generating function for ISFs of $G$ is given by the characteristic polynomial of $\AA(G)$ precisely when the vertices of $G$ are labeled by a PEO.  Combined with well-known theorems of Orlik--Solomon and Zaslavsky, this result has consequences for, respectively, Betti numbers of complex multigraph arrangements, Corollary~\ref{betti}, and for counting regions in their real versions, Corollary~\ref{region}.

In Section~\ref{inv}, we study the related class of \emph{tight forests}, in which every path starting at the root avoids the permutation patterns 231, 312, and 321.  The permutations avoiding these patterns are a special class of involutions that  we call \emph{tight involutions}, which are of independent combinatorial interest~\cite{DDJSS}.  Tight forests play an analogous role for triangle-free graphs as ISFs do  for general graphs.  Specifically, if $G$ is a triangle-free graph, then every tight spanning forest is an NBC~set, and the converse is true if the vertex labeling is a \emph{quasi-perfect ordering} or QPO (Definition~\ref{def:QPO}), a variation of the usual definition of a PEO.  Thus the existence of a QPO implies that the chromatic polynomial is a generating function for tight forests as shown in Theorem~\ref{orderingIffNBC}. Structurally, graphs with QPOs satisfy a property analogous to chordal graphs: every cycle of length 5 or greater has a chord, Proposition~\ref{chord}.  These results raise the question of studying labeled forests avoiding other pattern families.


\section{The original theorems revisited}
\label{otr}

In this section we will give a new proof of Theorem~\ref{HS1}.  The starting point is the characterization of increasing forests given in Lemma~\ref{ik,jk}, which will enable us to generalize the theory from graphs to simplicial complexes in Section~\ref{sc}.
We will also prove a refinement of Theorem~\ref{HS2}.  The original proof used the theory of quotient posets developed in~\cite{hs:fcp}.  Our proof does not need those ideas but instead uses Whitney's classic description of the coefficients of the chromatic polynomial in terms of broken circuits.

\ble\label{ik,jk}
A graph $F$ is an increasing forest if and  only if it contains no pair of edges $ik,jk$ with $i,j<k$.
\ele
\bprf
($\Rightarrow$) Suppose that $F$ is an increasing forest which contains two edges $ik,jk$ with $i,j<k$. Let $r$ be the root of the component containing these edges.  Then either the unique path from $r$ to $i$ goes through $k$ or the unique path from $r$ to $j$ goes through $k$ (or both).  But in either case the path contains a descent, either from $k$ to $i$ or from $k$ to $j$.  This contradicts the fact that $F$ is increasing.

($\Leftarrow$) Suppose that $F$ is a graph with no such pair of edges.  Then $F$ must be acyclic, for if $F$ has a cycle $C$, then taking $k$ to be the largest vertex on $C$ and $i,j$ its two neighbors produces a contradiction.

To show that $F$ is increasing, we again assume the opposite.  Let $T$ be a component tree of $F$ with root $r$ such that $r=v_0,v_1,\dots,v_\ell$ is a non-increasing path starting at $r$.  If we choose such a path of minimal length, then $v_0,\dots,v_{\ell-1}$ is an increasing path, but then $v_{\ell-2},v_\ell<v_{\ell-1}$, a contradiction.
\eprf
It is now a simple matter to prove Theorem~\ref{HS1}.

\medskip

\noindent {\em Proof (of Theorem~\ref{HS1}).}
The coefficient of $t^{n-m}$ in $\prod_{k=1}^n (t+|E_k|)$ counts the number of graphs formed by picking $m$ edges of $G$ with at most one from each $E_k$.  By Lemma~\ref{ik,jk}, these graphs are exactly the increasing spanning forests of $G$.
\hqed

\vs{15pt}

Theorem~\ref{HS1} admits a weighted generalization, as follows.  Let $\bx=\{x_e\ |\ e\in E(G)\}$ be a set of commuting indeterminates.  Associate with any spanning subgraph $H\sbe G$ the monomial
$$
\bx_G =\prod_{e\in E(H)} x_e
$$
and define a generating function
$$
\ISF(G,t,\bx)=\sum_F \bx_F t^{n-|E(F)|}
$$
where the sum runs over all increasing spanning forests $F$ of $G$.  Clearly substituting $1$ for each $x_e$ in this polynomial recovers the original $\ISF(G,t)$.  The same proof given above also demonstrates the following result.
\bth\label{ISF-weighted-factorization}
Let $G$ be a graph with $V=[n]$.  Then
$$
\ISF(G,t,\bx)=\prod_{k=1}^n (t + E_k(\bx))
$$
where $E_k(\bx)=\sum_{e\in E_k} x_e$.\hqed
\eth

In order to generalize Theorem~\ref{HS2}, we need to review properties of broken circuits.  Assume that the edges of $G$ have been given a total order $e_1<e_2<\dots<e_p$.  A {\em broken circuit}  of $G$ is an edge set $\brk(C)$ obtained from a cycle $C$ by removing its smallest edge.  An {\em NBC~set} is an edge set containing no broken circuit.  Note in particular that every NBC~set is acyclic, since if $F\supseteq C$ then $F\supseteq\brk(C)$.
Set
\begin{align*}
\cNBC(G) &= \text{the set of NBC subsets of $E(G)$}, & \nbc(G)&=|\cNBC(G)|, \\
\cNBC_m(G)&=\text{the set of NBC subsets of $E(G)$ with $m$ edges}, & \nbc_m(G)&=|\cNBC_m(G)|.
\end{align*}
The notation does not reflect the edge ordering, but it will be clear from context.  In fact the numbers $\nbc_m(G)$ do not depend on the choice of ordering:
\bth[Whitney's formula~\cite{whi:lem}]\label{Whitney}
Let $G$ be a graph with $n$ vertices.  Then
$$
P(G,t) = \sum_{m\ge0} (-1)^m \nbc_m(G) t^{n-m}
$$
regardless of the ordering of the edges of $G$.\hqed
\eth

We next identify the relationship between increasing spanning trees and NBC~sets.
The following result, together with Whitney's formula, immediately implies Theorem~\ref{HS2}.

\bth\label{isf-and-nbc}
Let $G$ be a graph with $V(G)=[n]$.  Order the edges of $G$ lexicographically.
For each $m\ge0$ we have
\beq
\label{ISFinNBC}
\cISF_m(G)\sbe \cNBC_m(G).
\eeq
Furthermore, the following statements are equivalent:
\begin{enumerate}[label=(\alph*)]
\item  $\cISF_m(G)= \cNBC_m(G)$ for all $m\ge0$,
\item  $\cISF_2(G) = \cNBC_2(G)$,
\item  the natural ordering $1,2,\dots,n$ is a PEO for $G$.
\end{enumerate}
\eth

\bprf
First we prove~\eqref{ISFinNBC}.  Let $F$ be an increasing spanning forest.
Suppose that $F$ contains a broken circuit $B$, which must be a path of the form $v_1,v_2,\dots,v_\ell$ with $\ell\geq 3$, $v_1=\min\{v_2,\dots,v_\ell\}$, and $v_2>v_\ell$.  Then there must exist a smallest index $p>1$ such that $v_p>v_{p+1}$, and in particular $v_{p-1},v_{p+1}<v_p$, contradicting the criterion of Lemma~\ref{ik,jk}.  Therefore $F$ is an NBC~set.

Next we show that conditions (a), (b), and (c) are equivalent.

(a)$\Rightarrow$(b): Trivial.

(b)$\Rightarrow$(c): Assume that (b) holds.  Let $i,j,k\in[n]$ with $i<j<k$ and suppose $ik,jk\in E(G)$.  The edge set $\{ik,jk\}$ is not an increasing forest, so by (b) it must contain a broken circuit $B$.  This forces $ij\in E(G)$ as the edge which was removed to form $B=\{ik,jk\}$.  Hence we have established the second condition in Definition~\ref{def:PEO}.

(c)$\Rightarrow$(a): Assume that (c) holds.  Let $F$ be an NBC set; in particular it is a forest.  If it is not increasing, then by Lemma~\ref{ik,jk} it contains two edges $ik,jk$ with $i<j<k$.  But if $ij\in E(G)$ then these two edges from a broken circuit, while if $ij\not\in E(G)$ then the vertex ordering is not a PEO.  In either case we have a contradiction, so $F$ is an increasing spanning forest.
\eprf

Stanley \cite{Stanley-AO} discovered a fundamental relationship between acyclic orientations and the chromatic polynomial, whose best known special case is as follows.

\bth[\cite{Stanley-AO}]
The number of acyclic orientations of $G$ is $\ao(G)=(-1)^{|V(G)|} P(G,-1)$.
\eth
Thus $\ao(G)=\nbc(G)$ by Whitney's formula (Theorem~\ref{Whitney}), and combining these results with Theorem~\ref{isf-and-nbc} immediately yields the following corollary.

\bco \label{isf-ao}
For every graph $G$, we have $\isf(G)\leq\ao(G)$, with equality if and only if the labeling is a PEO.
\eco

Blass and Sagan \cite{BS} constructed a bijection (actually, a family of bijections) between acyclic orientations and NBC~sets of any given graph $G$.  This correspondence, together with the fact that every increasing spanning forest is an NBC~set by Theorem~\ref{isf-and-nbc}, gives a combinatorial explanation of Corollary~\ref{isf-ao}.  (Note that \cite{BS} uses the convention that a broken circuit is obtained by deleting the \emph{largest} edge of a cycle.)


\section{Simplicial complexes}
\label{sc}

In this section we will generalize Theorems~\ref{HS1} and~\ref{HS2} from graphs to simplicial complexes.  
For general background on simplicial complexes, see, e.g., \cite{GreenBook}.
Throughout, we let $\De$ be a pure simplicial complex of dimension~$d\geq 1$, with vertices $V=V(\De)=[n]$.  The symbol $\De_k$ denotes the set of simplices in $\De$ of dimension $k$,
and $\tilde H_k(\De)$ denotes reduced simplicial homology with coefficients in $\bbZ$.  
A subcomplex $\Upsilon\subseteq\Delta$ is a \emph{spanning subcomplex} if it contains all faces of $\Delta$ of dimension $<d$.  (Note that $\Upsilon$ need not be pure.)
Faces of dimensions $d$, $d-1$ and $d-2$ are called \emph{facets}, \emph{ridges}, and \emph{peaks}, respectively.    The notation $\langle\phi_1,\dots,\phi_n\rangle$ indicates the complex generated by the $\phi_i$.  The \emph{link} of a face $\sigma\in\Delta$ is $$
\link_\Delta(\sigma)=\{\tau~|~\tau\cap\sigma=\emptyset,\ \tau\cup\sigma\in\Delta\}.  
$$
Note that $\dim\link_\Delta(\sigma)=\dim\Delta-\dim\sigma-1$. 

We use the notation $[i_0,i_1,\dots,i_\ell]_<$ to indicate the simplex with vertices $i_0<i_1<\cdots<i_\ell$.  We extend this notation to simplices obtained by adjoining vertices to smaller simplices as follows.  If $\sigma$ is an $\ell$-simplex and $i$ is a vertex, then we write $\sigma<i$ to mean that $v<i$ for all $v\in\sigma$, and we denote the $(\ell+1)$-simplex $\sigma\cup\{i\}$ by the symbol $[\sigma,i]_<$.  Similar extensions should be self explanatory,  for instance $[\sigma,i,j]_<$ denotes the $(\ell+2$)-simplex $\sigma\cup\{i,j\}$, where $\sigma<i<j$.  When the vertices in a face are explicit positive integers, we will abbreviate
the simplex to a sequence.  For example, $[1,3,4,6]_<$  will be written $1346$.

A pure simplicial complex of dimension~1 is just a graph with no isolated vertices.  Note that in Section~\ref{otr}, we permitted graphs to contain isolated vertices; however, these have little effect on the polynomials under consideration --- if $G$ is obtained from $H$ by introducing an isolated vertex, then $\ISF(G,t)=t\ISF(H,t)$ and $P(G,t)=t P(H,t)$.  So the polynomials considered in this section will merely differ by a power of~$t$ from those introduced before.

The first step is to generalize the characterization of increasing spanning forests (Lemma~\ref{ik,jk}) to higher dimension.  First we introduce some terminology.

\begin{defn} \label{caged}
Let $\Delta$ be a simplicial complex.
A ridge $\rho=[\si,k]_<$ is {\em caged} if $\De$ contains two facets of the form $\phi_1=[\si,i,k]_<$ and $\phi_2=[\si,j,k]_<$.
\end{defn}

We use the term ``caged" because we regard $\rho$ as being ``trapped'' between the facets $\phi_1$ and $\phi_2$.  Note that in a graph, vertex $k$ is caged if and only if it satisfies the edge-pair criterion of Lemma~\ref{ik,jk}.

As a running example, consider the simplicial complex $\Delta=\langle 123, 124, 134\rangle$ shown in Figure~\ref{complexes}.  Ridge $14$ is caged by facets $124$ and $134$, but ridge $13=[1,3]_<$ is not caged because $\Delta$ has only one facet of the form $[1,j,3]_<$, namely $123$.

\bfi
\begin{tikzpicture}
\newcommand{\vxsize}{0.1}
\coordinate (v1) at (0,0);
\coordinate (v2) at (0,2);
\coordinate (v3) at (-1.73,-1);
\coordinate (v4) at (1.73,-1);
\draw[fill=light] (v3)--(v4)--(v2)--cycle;
\foreach \v in {v1,v2,v3,v4} \fill(\v) circle(\vxsize);
\draw(0,-.5) node{$1$};
\draw(0,2.5) node{$2$};
\draw(-2.25,-1) node{$3$};
\draw(2.25,-1) node{$4$};
\draw[thick] (v3)--(v4)--(v2)--(v3)--(v1)--(v4)  (v1)--(v2);
\draw(0,-2) node{\Large$\De$};

\begin{scope}[shift={(7,0)}]
\coordinate (v1) at (0,0);
\coordinate (v2) at (0,2);
\coordinate (v3) at (-1.73,-1);
\coordinate (v4) at (1.73,-1);
\draw [fill=light] (v3)--(v1)--(v4)--(v2)--cycle;
\foreach \v in {v1,v2,v3,v4} \fill(\v) circle(\vxsize);
\draw(0,-.5) node{$1$};
\draw(0,2.5) node{$2$};
\draw(-2.25,-1) node{$3$};
\draw(2.25,-1) node{$4$};
\draw[thick] (v3)--(v4)--(v2)--(v3)--(v1)--(v4)  (v1)--(v2);
\draw(0,-2) node{\Large$\Upsilon$};
\end{scope}
\end{tikzpicture}
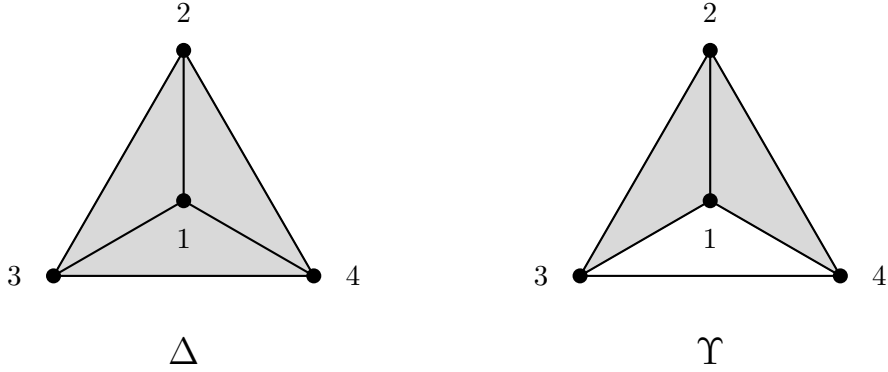
\caption{A simplicial complex $\De$ and a cage-free spanning subcomplex $\Upsilon$.\label{complexes}}
\efi

\begin{defn}
A spanning subcomplex $\Up\subseteq\Delta$ is \emph{cage-free} if it contains no caged ridges. 
\end{defn}

For example, if $\Delta=\langle 123,124,134\rangle$ as in Figure~\ref{complexes}, then the spanning subcomplex $\Upsilon=\langle 123,124,24\rangle$ is cage-free. 

As in~\eqref{ISF-notation}, we introduce the notation
\begin{equation} \label{CF-notation}
\begin{aligned}
\cCF(\Delta) &=\text{set of cage-free subcomplexes of $\Delta$}, & \cf(\Delta) &= |\cCF(\Delta)|,\\
\cCF_m(\Delta)&=\text{set of cage-free subcomplexes of $\Delta$ with $m$ facets}, & \cf_m(\Delta) &= |\cCF_m(\Delta)|,.
\end{aligned}
\end{equation}

When $d=1$ (i.e., $\Delta$ is a graph), this condition specializes to that of Lemma~\ref{ik,jk}: a spanning subcomplex $\Upsilon\subseteq\Delta$ is cage-free precisely if it is an increasing spanning forest.  Cage-free subcomplexes generalize spanning forests in the following additional ways.

\bpr\label{CFfacts}
Let $\Up$ be a cage-free simplicial complex on $V=[n]$ of dimension $d\ge1$.  Then
\begin{enumerate}[label=(\alph*)]
\item $\tilde H_d(\Up)=0$; and
\item If $\Upsilon_d\neq\0$, then $\Upsilon$ has at least one \emph{leaf}, i.e., a ridge contained in exactly one facet.
\end{enumerate}
\epr
\bprf
(a) Suppose, to the contrary, that $\Up$ contains a $d$-cycle $Z$.  Let $k$ be the maximum vertex contained in a $d$-simplex in $Z$.  Choose a ridge of $\Up$ of the form $\rho=[\rho',m,k]_<$, where the vertex $m$ is as small as possible.  We claim that every facet of $Z$  containing $\rho$ must be of the form $[\rho',m,i,k]_<$.  If there is a facet not of this form then, by maximality of $k$, it must have the form $[\si,h,m,k]_<$ for some $\si$ and $h$ where $h<m$.  But then $[\si,h,k]_<$ is a ridge of $Z$ containing $k$ with $h<m$, contradicting the choice of $\rho$ and proving the claim.  Furthermore, $Z$ is a cycle so that it must have at least two facets containing $\rho$.  And by the claim, these two facets cage $\rho$, a contradiction.

(b) This assertion is obtained by replacing $Z$ with $\Upsilon$ in the proof of (a).
\eprf

Assertion (a) of Proposition~\ref{CFfacts} specializes to acyclicity for graphs, and assertion (b) generalizes the statement that every forest with at least one edge has a leaf.

To define the appropriate generating function for cage-free subcomplexes of~$\De$, we will need to consider sets analogous to the edge sets $E_k$ defined in~\ree{E_k}.  For a peak $\si\in\De_{d-2}$ and a vertex $k>\sigma$, define
\beq
\label{Phi_si,k}
\Phi_{\si,k} =\Phi_{\si,k}(\De)=\{\phi\in\De_d\ :\ \text{$\phi=[\si,j,k]_<$ for some $j$ with $\sigma<j<k$}\}.
\eeq
Let $N=N(\Delta)$ be the number of nonempty sets $\Phi_{\si,k}$.  In our previous example, the nonempty sets are
$$
\Phi_{1,3}=\{123\},\ \Phi_{1,4}=\{124, 134\},
$$
and so $N=2$.  Note that the $\Phi_{\si,k}$ are pairwise-disjoint and partition $\De_d$.,  Moreover, the ridge $\rho=[\sigma,k]_<$  is caged in $\Delta$  if and only if $|\Phi_{\si,k}|\ge2$. 

Let $\bx=(x_\phi)$ be a family of commuting variables indexed by facets $\phi\in\Delta_d$.
For each set $\Phi_{\sigma,k}\subset\Delta_d$, define
$$
\Phi_{\si,k}(\bx)=\sum_{\phi\in\Phi_{\si,k}}x_\phi.
$$
For each cage-free spanning subcomplex $\Upsilon\subset\De$, define a monomial
$$
\bx_\Up=\prod_{\phi\in\Up_d} x_\phi.
$$
Define generating functions
$$
\CF(\De,t,\bx)=\sum_\Upsilon \bx_\Upsilon t^{N-|\Upsilon_d|}, \qquad
\CF(\De, t)  = \CF(\De,t,\bx)\big\vert_{x_\phi=1} = \sum_{\Upsilon} t^{N-|\Upsilon_d|},
$$
where $N$ is the number of nonempty sets $\Phi_{\si,k}$ and both sums run over all cage-free subcomplexes $\Up\sbe\De$.     Our example complex has 
\begin{align*}
\CF(\De,t,\bx) &= t^2 +(x_{123}+x_{124}+x_{134})t + (x_{123} x_{124} + x_{123} x_{134})\\
&= (t+x_{123})(t+ x_{124}+x_{134})\\
&= (t+\Phi_{1,3}(\bx)) (t+\Phi_{1,4}(\bx)),\\
\CF(\De,t) &= t^2 +3t+2 = (t+1)(t+2).
\end{align*}

We now have all the pieces in place to state a factorization theorem generalizing Theorem~\ref{ISF-weighted-factorization} (which is the special case $d=1$).

\bth\label{CFfactor}
Let $\De$ be a simplicial complex on $V=[n]$ of dimension  $d\ge1$.  Then
$$
\CF(\De,t,\bx)=\prod_{\si,k}  (t+\Phi_{\si,k}(\bx))
$$
where the product is over all $\si,k$ such that $\Phi_{\si,k}(\De)\neq\emp$.
\eth
\bprf
It follows directly from the definitions that $\Up$ is cage-free precisely when the elements of $\Up_d$ are obtained by picking at most one facet from each $\Phi_{\si,k}$.  Translating this statement into generating functions gives the desired equality.
\eprf

\begin{figure}
\begin{center}
\begin{tikzpicture}
\newcommand{\hsh}{4} 
\coordinate (v1) at (0,0);
\coordinate (v2) at (-.3,-1.2);
\coordinate (v3) at (0,-2);
\coordinate (v4) at (1,-1);
\coordinate (v5) at (-1,-1);
\draw[thick,fill=light] (v1) -- (v5) -- (v2) -- cycle;
\draw[thick,fill=kindoflight] (v1) -- (v4) -- (v2) -- cycle;
\draw[thick,fill=light] (v3) -- (v4) -- (v2) -- cycle;
\draw[thick,fill=kindoflight] (v3) -- (v5) -- (v2) -- cycle;
\draw[thick,dashed] (-1,-1)--(1,-1);
\foreach \v in {v1,v2,v3,v4,v5} \fill(\v) circle (.1);
\draw(0,.4) node{\footnotesize 1};
\draw(0,-1.4) node{\footnotesize 2};
\draw(0,-2.4) node{\footnotesize 3};
\draw(1.4,-1) node{\footnotesize 4};
\draw(-1.4,-1) node{\footnotesize 5};
\draw(0,-3.2) node {\large $\Delta$};

\begin{scope}[shift={(\hsh,0)}]
\coordinate (v1) at (0,0);
\coordinate (v2) at (-.3,-1.2);
\coordinate (v3) at (0,-2);
\coordinate (v4) at (1,-1);
\coordinate (v5) at (-1,-1);
\foreach \v in {v1,v2,v3,v4,v5} \fill(\v) circle (.1);
\draw(0,.4) node{\footnotesize 1};
\draw(0,-1.4) node{\footnotesize 2};
\draw(0,-2.4) node{\footnotesize 3};
\draw(1.4,-1) node{\footnotesize 4};
\draw(-1.4,-1) node{\footnotesize 5};
\draw (v2) -- (v4) -- (v5) -- cycle;
\draw(0,-3.2) node {\large $G_1$};
\end{scope}

\begin{scope}[shift={(2*\hsh,0)}]
\coordinate (v1) at (0,0);
\coordinate (v2) at (-.3,-1.2);
\coordinate (v3) at (0,-2);
\coordinate (v4) at (1,-1);
\coordinate (v5) at (-1,-1);
\foreach \v in {v1,v2,v3,v4,v5} \fill(\v) circle (.1);
\draw(0,.4) node{\footnotesize 1};
\draw(0,-1.2) node{\footnotesize 2};
\draw(0,-2.4) node{\footnotesize 3};
\draw(1.4,-1) node{\footnotesize 4};
\draw(-1.4,-1) node{\footnotesize 5};
\draw (v5)--(v3)--(v4);
\draw(0,-3.2) node {\large $G_2$};
\end{scope}

\begin{scope}[shift={(3*\hsh,0)}]
\coordinate (v1) at (0,0);
\coordinate (v2) at (-.3,-1.2);
\coordinate (v3) at (0,-2);
\coordinate (v4) at (1,-1);
\coordinate (v5) at (-1,-1);
\foreach \v in {v1,v2,v3,v4,v5} \fill(\v) circle (.1);
\draw(0,.4) node{\footnotesize 1};
\draw(0,-1.2) node{\footnotesize 2};
\draw(0,-2.4) node{\footnotesize 3};
\draw(1.4,-1) node{\footnotesize 4};
\draw(-1.4,-1) node{\footnotesize 5};
\draw (v5) -- (v4);
\draw(0,-3.2) node {\large $G_3$};
\end{scope}
\end{tikzpicture}
\end{center}
\caption{A bipyramid $\De$ and its upper links $G_1,G_2,G_3$.\label{bipyFig} }
\end{figure}
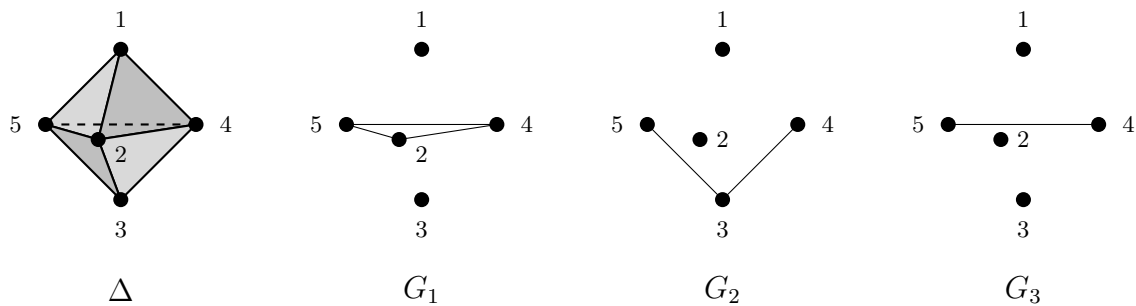

Next we show that, up to a correction factor, the generating function $\CF(\De,t)$ is in fact the product of generating functions $\ISF(G_\sigma,t)$ for a family of graphs associated with~$\Delta$.

\begin{defn}
Let $\sigma$ be a peak (a codimension-2 face) of $\Delta$.  The \emph{upper link} of $\sigma$ is
the graph $G_\sigma$ on $[n]$ with edges $\{ij \mid [\sigma, i,j]_<\in \Delta\}$.
The peak $\sigma$ is called \emph{effective} if $G_\sigma$ has at least one edge.
\end{defn}

For example, let $\Delta$ be the triangular bipyramid with facets $124,125,145, 234, 235, 345$.  Figure~\ref{bipyFig} illustrates $\De$ and the upper links of the effective peaks, namely the vertices 1, 2, and 3.  Moreover,
\begin{align*}
\CF(\Delta, t) &= (t+1)^4(t+2)\\ &= t^{-10}\ISF(G_1,t)\ISF(G_2,t)\ISF(G_3,t).
\end{align*}
The correction factor $t^{-10}$ arises because $10=15-5$ is the difference between the degree of the product of the $\ISF(G_\si,t)$ and the degree of $\CF(\Delta, t)$.

This factorization is an instance of the following general statement.

\begin{prop}\label{CFdeltaEqCFgraph}
Let $\Delta$ be a simplicial complex on $n$ vertices.  Let $N$ be the number of nonempty sets $\Phi_{\sigma,k}$
 and let $s$ be the number of effective peaks of $\Delta$.  Then:
\begin{enumerate}[label=(\alph*)]
\item
The cage-free generating function is given by
$$
\CF(\Delta, t)  = t^{N-ns}\prod_{\text{effective peaks }\sigma} \ISF(G_\sigma, t).
$$
\item The number of cage-free subcomplexes of $\De$ is given by
$$
\cf(\De)  = \prod_{\sigma} \isf(G_\sigma).
$$
\end{enumerate}
\end{prop}
\bprf
First note that (b) follows from (a) by setting $t=1$.  To prove (a), 
let $G$ be the disjoint union of all graphs $G_\sigma$, where $\sigma$ is an effective peak.
By ``disjoint" we mean that when the same edge of $\De$  occurs in many $G_\si$, the different copies are considered distinct in $G$.
Since
$$
\ISF(G,t)  = \prod_\sigma \ISF(G_\sigma, t)
$$
it is enough to show that
$$
\CF(\Delta, t)  = t^{N-ns}\ISF(G,t) .
$$

Each facet $\phi$ can be written uniquely as $[\sigma,i,j]_<$, and so gives rise to a unique edge $ij=\gamma(\phi)\in E(G)$
where we are considering $ij$ as an edge of $G_\si$.  This map $\gamma:\Delta_d\to E(G)$ is bijective because it has an inverse: given $ij\in E(G_\si)$ then the corresponding facet is $[\sigma,i,j]_<$.   Now, a spanning subcomplex $\Upsilon\subseteq\Delta$ is cage-free if and only if the corresponding edge set $\gamma(\Upsilon)$ contains no two edges $ik,jk$ with $i,j<k$.   And by Lemma~\ref{ik,jk} this is precisely the statement that $\gamma(\Upsilon)$ is the edge set of an ISF of $G$.

It follows that $\CF(\De,t)$ and $\ISF(G,t)$ are equal up to multiplying by a power of $t$.   By Theorem~\ref{CFfactor},  $\deg \CF(\De,t)=N$, and $\deg\ISF(G,t)=ns$.  So $t^{N-ns}$ is the appropriate correction factor.
\eprf

Next we discuss simplicial extensions of the concept of a perfect elimination ordering.

\begin{defn} \label{simplicial-PEO}
Let $\Delta$ be a pure simplicial complex of dimension $d\geq 1$ with vertices $1,\dots,n$.   The labeling is a \emph{perfect elimination ordering (PEO)} if for all $(d-2)$-faces $\sigma$ and vertices $i,j,k$ with $\sigma<i<j<k$, we have
$$
[\sigma, i,k]_<,\ [\sigma,j,k]_< \in \Delta \quad\Rightarrow\quad [\sigma, i, j]_<\in\Delta.
$$
\end{defn}

If $\De$ is a graph, then $\sigma=\emptyset$ and so Definition~\ref{simplicial-PEO} reduces to 
the definition of a PEO of a graph.  Recall our example of the bipyramid (see Figure~\ref{bipyFig}) with facets, $124, 125,145, 234, 235, 345$. It is easy to directly check that this labeling is a PEO of the bipyramid.

It is not hard to see that a labeling of the vertex set of $\Delta$ is a PEO if and only if the induced labeling of each $G_\sigma$ is a PEO.  Together with Theorem~\ref{HS2}, Corollary~\ref{isf-ao}, and Proposition~\ref{CFdeltaEqCFgraph}, we get the following corollary.

\begin{cor} \label{more-about-PEO}
Let $\Delta$ be a simplicial complex on $n$ vertices.  
\begin{itemize}
\item[(a)] As above, let $N$ be the number of nonempty sets $\Phi_{\sigma,k}$ and let $s$ be the number of effective peaks.  Then
$$
\CF(\Delta, t)  =(-1)^{ns} t^{N-ns}\prod_{\sigma} P(G_\sigma, -t)
$$
where  the product is over all peaks $\sigma$ such that $[\sigma, i,j]_< \in \Delta$ if and only if the ordering $1,2,\dots, n$ is a PEO of $\Delta$.
\item[(b)] We have the following relationship between cage-free subcomplexes of $\De$ and acyclic orientations:
$$
\cf(\De) \leq \prod_{\sigma} \ao(G_\sigma) 
$$
with equality if and only if the ordering $1,2,\dots, n$ is a PEO of $\Delta$.
\end{itemize}
\end{cor}

Which simplicial complexes have perfect elimination orderings?  It is well known that a graph has a PEO if and only if it is chordal, but the situation in higher dimension is much more complicated.
Higher-dimensional extensions of various equivalent characterizations of chordality, have been studied by, e.g., H\`{a} and Van~Tuyl \cite{HVT}, Emtander \cite{Emt}, Woodroofe \cite{Woodroofe}, and Adiprasito, Nevo and Samper \cite{ANS}.

A simplicial complex $\Delta$ on vertex set $[n]$ is called \emph{shifted} if, whenever $\sigma\in\Delta$ is a face and $j<k$ with $j\not\in\Delta$ and $k\in\Delta$, then $\sigma\sm\{k\}\cup\{j\}$ is a face.  Shifted complexes of dimension~1 are called \emph{threshold graphs}; both these classes are  well known in combinatorics.
The vertex labeling on a shifted complex is always a PEO, but not every complex with a PEO is shifted.  For instance, this is true of the bipyramid of Figure~\ref{bipyFig}, which cannot be made shifted even by relabeling the vertices.  In dimension~$1$, this observation reduces to the statement that the threshold graphs are a proper subset of the chordal graphs.

The following is a connection between our notion of a PEO and chordality of graphs.
\begin{prop} \label{has-chordal-link}
Let $\Delta$ be a simplicial complex.  If $\Delta$ has a PEO, then there is some peak $\sigma$ whose link is a chordal graph.
\end{prop}
\begin{proof}
Let $\sigma$ be the lexicographically smallest peak, and let $ij$ be an edge with $i<j$.
If $ij\in\link_\Delta(\sigma)$ then $\sigma<i<j$, else we could replace  the greatest vertex of $\sigma$ with $i$ to obtain a lexicographically smaller peak.  Therefore $\link_\Delta(\sigma)$ coincides with the upper link $G_\sigma$ up to isolated vertices, and is chordal by the remarks preceding Corollary~\ref{more-about-PEO}.
\end{proof}

Note that while the bipyramid over a triangle has a PEO as we have seen in  Figure~\ref{bipyFig}, no bipyramid over a polygon with more than three sides has a PEO.  This is because peaks are vertices, and every link of a vertex is a cycle of length at least $4$ which is not a chordal graph.  
In particular, whether a simplicial complex has a PEO cannot be determined by its topology which is true even in dimension~$1$.  This example also illustrates that shellability does not imply the existence of a PEO.  Indeed, in dimension~$1$, any cycle  is shellable but not chordal if it has length at least $4$.  Neither is the converse true: the ``bowtie'' complex consisting of two triangles joined at a vertex is not shellable, but every labeling is trivially a PEO since each ridge belongs to only one facet.

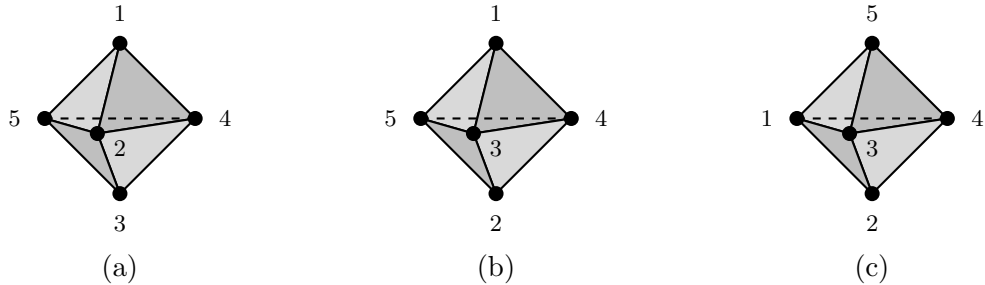
\begin{figure}
\begin{center}
\begin{tikzpicture}[scale=1]
\coordinate (v1) at (0,0);
\coordinate (v2) at (-.3,-1.2);
\coordinate (v3) at (0,-2);
\coordinate (v4) at (1,-1);
\coordinate (v5) at (-1,-1);
\draw[thick,fill=light] (v1) -- (v5) -- (v2) -- cycle;
\draw[thick,fill=kindoflight] (v1) -- (v4) -- (v2) -- cycle;
\draw[thick,fill=light] (v3) -- (v4) -- (v2) -- cycle;
\draw[thick,fill=kindoflight] (v3) -- (v5) -- (v2) -- cycle;
\draw[thick,dashed] (-1,-1)--(1,-1);
\foreach \v in {v1,v2,v3,v4,v5} \fill(\v) circle (.1);
\draw(0,.4) node{\footnotesize 1};
\draw(0,-1.4) node{\footnotesize 2};
\draw(0,-2.4) node{\footnotesize 3};
\draw(1.4,-1) node{\footnotesize 4};
\draw(-1.4,-1) node{\footnotesize 5};
\draw(0,-3) node {(a)};

\begin{scope}[shift={(5,0)}]
\coordinate (v1) at (0,0);
\coordinate (v2) at (-.3,-1.2);
\coordinate (v3) at (0,-2);
\coordinate (v4) at (1,-1);
\coordinate (v5) at (-1,-1);
\draw[thick,fill=light] (v1) -- (v5) -- (v2) -- cycle;
\draw[thick,fill=kindoflight] (v1) -- (v4) -- (v2) -- cycle;
\draw[thick,fill=light] (v3) -- (v4) -- (v2) -- cycle;
\draw[thick,fill=kindoflight] (v3) -- (v5) -- (v2) -- cycle;
\draw[thick,dashed] (-1,-1)--(1,-1);
\foreach \v in {v1,v2,v3,v4,v5} \fill(\v) circle (.1);
\draw(0,.4) node{\footnotesize 1};
\draw(0,-1.4) node{\footnotesize 3};
\draw(0,-2.4) node{\footnotesize 2};
\draw(1.4,-1) node{\footnotesize 4};
\draw(-1.4,-1) node{\footnotesize 5};
\draw(0,-3) node {(b)};
\end{scope}

\begin{scope}[shift={(10,0)}]
\coordinate (v1) at (0,0);
\coordinate (v2) at (-.3,-1.2);
\coordinate (v3) at (0,-2);
\coordinate (v4) at (1,-1);
\coordinate (v5) at (-1,-1);
\draw[thick,fill=light] (v1) -- (v5) -- (v2) -- cycle;
\draw[thick,fill=kindoflight] (v1) -- (v4) -- (v2) -- cycle;
\draw[thick,fill=light] (v3) -- (v4) -- (v2) -- cycle;
\draw[thick,fill=kindoflight] (v3) -- (v5) -- (v2) -- cycle;
\draw[thick,dashed] (-1,-1)--(1,-1);
\foreach \v in {v1,v2,v3,v4,v5} \fill(\v) circle (.1);
\draw(0,.4) node{\footnotesize 5};
\draw(0,-1.4) node{\footnotesize 3};
\draw(0,-2.4) node{\footnotesize 2};
\draw(1.4,-1) node{\footnotesize 4};
\draw(-1.4,-1) node{\footnotesize 1};
\draw(0,-3) node {(c)};
\end{scope}
\end{tikzpicture}
\end{center}
\caption{Three different labelings of the bipyramid.\label{threeBips}}
\end{figure}

For a graph, Corollary~\ref{isf-ao} gives an upper bound for the number of increasing spanning forests and shows that the upper bound is achieved by labeling the vertices with a PEO.  In particular, the number of increasing spanning forests does not depend on the choice of PEO, and a PEO could be defined as a labeling which maximizes the number of increasing spanning forests.  These properties are not in general true for PEOs of a simplicial complex.  Consider the three labelings of the bipyramid shown in Figure~\ref{threeBips}.  Labelings (a) and (b) are PEOs, but not (c) because $135$ and $145$ are simplices, but not $134$.  On the other hand, labelings (a) and (c) give rise to the same cage-free generating function, namely $CF(\Delta,t)=(t+1)^4(t+2)$, while for labeling (b) one has instead $(t+1)^2(t+2)^2$.  The number of cage-free spanning subcomplexes is in fact maximized by both labelings (a) and (c).

If the top homology of a graph is trivial, then it is a forest, hence is chordal and has a PEO.  In higher dimension, vanishing top homology does not guarantee existence of a PEO.  For example, the dunce hat is contractible, hence acyclic, but one can check that no labeling of the eight-vertex triangulation of the dunce hat given in~\cite{Hach} is a PEO.  Note that the link of the vertex labeled 1 in~\cite{Hach} is a chordal graph, so that the dunce hat is a counterexample to the converse of Proposition~\ref{has-chordal-link}.

\begin{question}
Can one classify all simplicial complexes which have a PEO? 
\end{question}


\section{ISFs in multigraphs}
\label{multi}

In this section we generalize the theory of increasing spanning trees from graphs to multigraphs.

\begin{defn}
Let $n$ be a positive integer.  A \emph{labeled $n$-multigraph} is a multigraph $G=(V,E)$ such that:
\begin{enumerate}[label=(\alph*)]
\item $V=\{0,1,\dots,n\}$;
\item $G$ has no loops, and at most one edge $0k$ for each $k\in[n]$;
\item For $1\leq i<j\leq n$, each edge between~$i$ and $j$ is labeled with a nonzero complex number $\zeta$ and denoted by $ij^\zeta$.  No two edges with the same endpoints can have the same label.
\end{enumerate}
\end{defn}

We retain the notation~\eqref{ISF-notation} for ISFs of a labeled multigraph~$G$.  We also define
\begin{equation}\label{E_kMulti}
E_k = E_k(G) = \{e\in E(G) : e=jk^\gamma \text{ and } j, k \neq 0\} \cup \{e\in E(G) : e=0k\}.
\end{equation}
(cf.~\eqref{E_k}).

We will make the convention that two edges with the same endpoints form a cycle.  The characterization of increasing spanning forests (Lemma~\ref{ik,jk}) carries over to the setting of multigraphs, as does the factorization for the generating function $\ISF(G,t)$ (Theorem~\ref{HS1}).

For example, let $G$ be the labeled multigraph $G$ shown in Figure~\ref{multigraphFig}.  Then
\begin{equation} \label{edge-partn-example}
E_1 =\{01\},\quad  E_2 = \{12^\alpha, 12^\beta\},\quad E_3 = \{03,13^\gamma\}
\end{equation}
and
\[
\ISF(G,t) = t^3+5t^2+8t+4 = (t+1)(t+2)^2 = (t+|E_1|)(t+|E_2|)(t+|E_3|).
\]

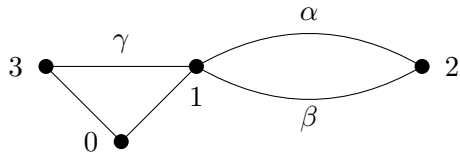
\begin{figure}[ht]
\begin{center}
\begin{tikzpicture}
\foreach \v in {(-1,1),(-2,2),(0,2),(3,2)} \fill \v circle(.1);
\draw(-1.4,1)  node{$0$};
\draw(-2.4,2) node{$3$};
\draw(0,1.6) node{$1$};
\draw(3.4,2) node{$2$};
\draw(0,2) -- (-2,2) -- (-1,1) -- (0,2) to[bend left] (3,2) to[bend left] (0,2);
\draw(1.5, 2.7) node {$\alpha$};
\draw(1.5, 1.3) node {$\beta$};
\draw(-1,2.3) node{$\gamma$};
\end{tikzpicture}
\caption{A labeled multigraph $G$.}\label{multigraphFig}
\end{center}
\end{figure}

We assume familiarity with the basic theory of posets, M\"obius functions, and hyperplane arrangements, as in Chapter~3 of~\cite{EC1}, and we will adopt the notation therein.  For convenience, we will refer to a hyperplane simply by its defining equation.

Let $G$ be a labeled multigraph on vertex set $\{0,\dots,n\}$.  Define a hyperplane arrangement in~$\Cc^n$ by
\begin{equation} \label{define-arrangement}
\AA(G) = \{ x_i=\gamma x_j ~|~ ij^\gamma \in E(G) \} \cup \{x_k=0 ~|~ 0k\in E(G)\}.
\end{equation}
This construction generalizes the usual one of a graphic hyperplane arrangement.
For example, if $G$ is the multigraph of Figure~\ref{multigraphFig}, then $\AA(G)$ is the arrangement in~$\Cc^3$ with five hyperplanes $x_1=0$, $x_1=\alpha x_2$, $x_1=\beta x_2$, $x_1 = \gamma x_3$, $x_3=0$.  The intersection lattice $L(G)=L(\AA(G))$ is shown in Figure~\ref{intLatFig}.

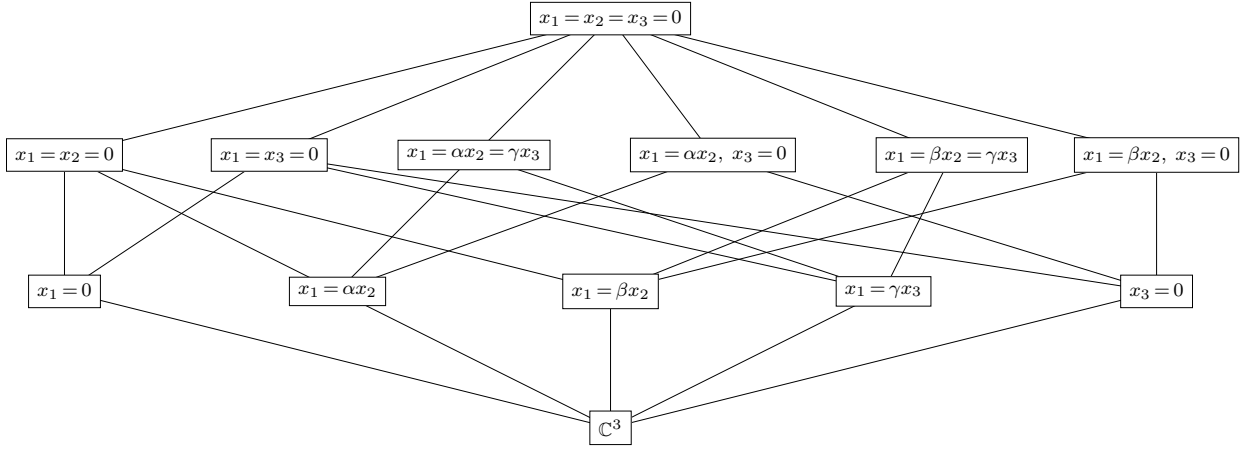
\begin{figure}[ht]
 \noindent\resizebox{\textwidth}{!}{\begin{tikzpicture}
\coordinate (A) at (-8,4);
\coordinate (B) at (-5,4);
\coordinate (C) at (-2,4);
\coordinate (D) at (1.5,4);
\coordinate (E) at (5,4);
\coordinate (F) at (8,4);
\foreach \xco in {-8,-4,0,4,8} \draw (0,0) -- (\xco,2);
\foreach \v in {A,B} \draw (-8,2) -- (\v);
\foreach \v in {A,C,D} \draw (-4,2) -- (\v);
\foreach \v in {A,E,F} \draw (0,2) -- (\v);
\foreach \v in {B,C,E} \draw (4,2) -- (\v);
\foreach \v in {B,D,F} \draw (8,2) -- (\v);
\foreach \v in {A,B,C,D,E,F} \draw (0,6) -- (\v);
\node[draw,rectangle,fill=white] at (0,0) { $\scriptstyle \mathbb{C}^3$};
\node[draw,rectangle,fill=white] at (-8,2) {$\scriptstyle x_1\,=\,0$};
\node[draw,rectangle,fill=white] at (-4,2) {$\scriptstyle x_1\,=\,\alpha x_2$};
\node[draw,rectangle,fill=white] at (0,2) {$\scriptstyle x_1\,=\,\beta x_2$};
\node[draw,rectangle,fill=white] at (4,2) {$\scriptstyle x_1\,=\,\gamma x_3$};
\node[draw,rectangle,fill=white] at (8,2) {$\scriptstyle x_3\,=\,0$};
\node[draw,rectangle,fill=white] at (A) {$\scriptstyle x_1\,=\,x_2\,=\,0$};
\node[draw,rectangle,fill=white] at (B) {$\scriptstyle x_1\,=\,x_3\,=\,0$};
\node[draw,rectangle,fill=white] at (C) {$\scriptstyle x_1\,=\,\alpha x_2\,=\, \gamma x_3$};
\node[draw,rectangle,fill=white] at (D) {$\scriptstyle x_1\,=\,\alpha x_2,\  x_3\,=\,0$};
\node[draw,rectangle,fill=white] at (E) {$\scriptstyle x_1\,=\,\beta x_2\,=\,\gamma x_3$};
\node[draw,rectangle,fill=white] at (F) {$\scriptstyle x_1\,=\,\beta x_2,\  x_3\,=\,0$};
\node[draw,rectangle,fill=white] at (0,6) {$\scriptstyle x_1\,=\,x_2\,=\,x_3\,=\,0$};
\end{tikzpicture} }
\caption{The intersection lattice $L(G)$ for the multigraph $G$ in Figure~\ref{multigraphFig}.}\label{intLatFig}
\end{figure}

Its characteristic polynomial is
$$
\chi(L(G),t) ~=~  t^3-5t^2 +8t-4 ~=~ (t-1)(t-2)^2 ~=~ (-1)^3\ISF(G,-t).
$$
It is not a coincidence that the two polynomials are related.  To explain the relationship, we need more about intersection lattices.

Let $L$ be a lattice.  Recall that the elements of~$L$ that cover $\hat{0}$ are called \emph{atoms}, and that a \emph{multichain} is a totally ordered multisubset of~$L$.  Let $C$ be an $\hat{0}$-$\hat{1}$ multichain~$C$, i.e., a multichain of the form $\hat{0}=z_0\leq z_1\leq\dots \leq z_n=\hat{1}$.  Then $C$ induces an ordered partition of the atoms into blocks $A_1,A_2,\dots, A_n$, namely
\begin{equation} \label{atom-induced}
A_i = \{\text{atoms } a ~\mid~ a\leq z_i \text{ and } a\not\leq z_{i-1}\}.
\end{equation}
Note that some blocks can be empty, e.g., if $z_{i-1}=z_i$ for some $i$.
For example, if $L$ is the lattice in Figure~\ref{intLatFig}, then the sturated chain
\[\Cc^3\quad\lessdot\quad x_1=0 \quad\lessdot\quad x_1=x_2=0\quad\lessdot\quad x_1=x_2=x_3=0\]
induces the atom partition
$$
A_1 = \{x_1=0\},  \quad A_2=\{x_1= \alpha x_2,\  x_1 =\beta x_2\}, \quad A_3 =\{x_1=\gamma x_3,\  x_3=0\}.
$$
which corresponds to the partition of $E(G)$ given in~\eqref{edge-partn-example}.

\begin{defn} \label{perfect-labeling}
A labeled $n$-multigraph $G$ is \emph{perfectly labeled} if for all nonzero $i<j<k$ the  following hold:
\begin{enumerate}
\item[(1)] If $G$ has edges $ik^\alpha$ and $jk^\beta$, then it also has an edge $ij^{\alpha/\beta}$.
\item[(2)] If $G$ has edges $jk^{\gamma}$ and $jk^{\epsilon}$ with $\gamma\neq \epsilon$, then it also has an edge $0j$.
\item[(3)] If $G$ has edges $jk^{\gamma}$ and $0k$, then it also has an edge $0j$.
\end{enumerate} 
\end{defn}
In each of these cases, the third edge corresponds to a hyperplane whose defining equation is implied algebraically by those of the first two edges.  For instance, the first condition says that if $\AA(G)$ contains the hyperplanes $x_i=\alpha x_k$ and $x_j=\beta x_k$, then it also contains the hyperplane $x_i/x_j=\alpha/\beta$, i.e., $x_i=(\alpha/\beta)x_j$.  Thus perfect labelings are the analogues of PEOs in the setting of labeled multigraphs.  Unlike the definition of a PEO, a perfect labeling is not simply an ordering of the vertices.  On the other hand, any PEO of a (simple) graph can be regarded as a perfect labeling by assigning all edges label~1.

For instance, one can check that the multigraph in Figure~\ref{multigraphFig} is perfectly labeled. 

\begin{thm}\label{ISFandChi}
Let $G$ be a labeled $n$-multigraph.  Let $L=L(G)$ and let $\rho(L)$ denote the rank of~$L$.  Then 
$$
\ISF(G,t) = (-1)^{\rho(L)}  t^{n-\rho(L)}\chi(L,-t)
$$
if and only if $G$ is perfectly labeled.
\end{thm}
\begin{proof}
Let $V_m$ be the vector space obtained by intersecting all hyperplanes of the form $x_j=0$ and $x_i =\alpha x_j$  where $i< j\leq m$.  Note that $V_0=\Cc^n=\hat{0}_L$ and $V_n=\hat{1}_L$.  Let $C$ be the $\hat{0}$-$\hat{1}$ multichain $\hat{0}=V_0\leq V_1\leq\cdots \leq V_n=\hat{1}$ and let $(A_1,A_2,\dots, A_n)$ be the partition of the atom set induced by~$C$.  Then combining
Theorem~18 and Lemma~19 of~\cite{hs:fcp}, one sees that
\begin{equation} \label{hs-factor-chi}
\chi(L,t) = t^{\rho(L)-n}\prod_{i=1}^n (t-|A_i|)
\end{equation}
if and only if for every $x\in L$ which is the join of two elements from $A_k$ there exists a $j$ such that there is a unique atom below $x$ in $A_j$.   This second statement is precisely the condition that $G$ is perfectly labeled.  
Moreover, for each $i$, the edges in $E_i(G)$ (see~\eqref{E_kMulti}) correspond to the hyperplanes in $A_i$, so $|E_i|=|A_i|$.  It follows that
$$
\ISF(G,t) = \prod_{i=1}^n (t+|E_i|) = (-1)^n\prod_{i=1}^n(-t-|A_i|).
$$
Thus by~\eqref{hs-factor-chi},
$$
\ISF(G,t) = \frac{(-1)^n}{(-t)^{\rho(L)-n}} \chi(L,-t)
$$
if and only if the labeling of $G$ is perfect.  The result now follows.
\end{proof}

We now briefly review some concepts related to NBC~sets of geometric lattices.  For details, see, e.g., Lectures~3 and~4 of \cite{s:aiha}.  Let $L$ be a geometric lattice with rank function $\rho$.  For $S\subseteq L$, the symbol $\vee S$ denotes the join of all elements in~$S$.  A set $S$ of atoms of $L$ is \emph{independent} if $\rho(\vee S) = |S|$, and is a \emph{circuit} if it is a minimal dependent set.  (These terms refer to the matroid naturally associated with~$L$.)  Note that if $K$ is a circuit, then $\join K=\join(K\sm\{a\})$ for any $a\in K$.  If we fix a total order on the atoms, then $S$ is a \emph{broken circuit} if it is obtained by removing the smallest atom from a circuit.  An \emph{NBC~set} of $L$ is a set which does not contain a broken circuit.  Rota~\cite[Prop.~1]{R} proved that
\begin{equation} \label{Rota}
\chi(L,t) =  \sum_{m\geq 0} (-1)^m \nbc_m(L)t^{\rho(L)-m}
\end{equation}
where $\nbc_m(L)$ is the number of nbc sets of $L$ with $m$ atoms.
(When $L$ is the lattice of flats of a graph, Rota's formula reduces to Whitney's formula (Theorem~\ref{Whitney}).)
Combining Rota's formula with Theorem~\ref{ISFandChi}, we see that if $G$ is a labeled $n$-multigraph, then $\isf_m(G) = \nbc_m(L(G))$ for all $m$ if and only if $G$ is perfectly labeled.

Let $L$ be any lattice and let $(A_1,A_2,\dots, A_{n})$ be a partition of the atom set induced by a $\hat{0}$-$\hat{1}$ multichain, as in~\eqref{atom-induced}.   An \emph{atomic transversal} is a set $T$ of atoms such that $|T\cap A_i|\leq 1$ for each~$i$.  By~\cite[Lemma 19]{hs:fcp}, every atomic transversal is independent.  Evidently,
$$
\sum_T t^{n-|T|}=\prod_{k=1}^n(t+|A_k|)
$$
where the sum is over atomic transversals of $L$.  It follows from
Theorem~\ref{ISFandChi} that if $L=L(G)$ where $G$  is a labeled $n$-multigraph then the number of atomic transversals of size~$m$ is precisely $\isf_m(G)$.  Moreover, atomic transversals are related to NBC~sets in the following way.

\begin{prop}\label{transNBC}
Let $L$ be a geometric lattice and let $C:\hat{0}=z_0\leq z_1\leq \cdots \leq z_n =\hat{1}$ be a  $\hat{0}$-$\hat{1}$ multichain in $L$.  Let $(A_1,A_2,\dots, A_n)$ be the partition of the atoms induced by $C$.  Fix a total ordering of the atoms so that if $a\in A_i$ and $b\in A_j$ with $i<j$, then $a$ precedes $b$.  If $T$ is an atomic transversal, then $T$ is an NBC~set.
\end{prop}

\begin{proof}
Every subset of a transversal is a transversal, so it is enough to show that no atomic transversal can be a broken circuit.

Suppose that $T$ is an atomic transversal which is also a broken circuit, say $T=K\setminus\{a\}$ where $K$ is a circuit and $a=\min(K)$.  In particular $|K|\geq3$ and $|T|\geq2$.  Let $A_i$ be the block containing~$a$.  As mentioned above, $T$ is independent, so $T\cap A_i$ must be nonempty, otherwise $K$ would also be an atomic transversal, hence independent.  Let $j= \max \{k \mid T\cap A_k \neq \emptyset\}$. Since $|T|\geq 2$ and $a=\min(K)$, it follows that $j>i$. Let $b$ be the unique element of $T\cap A_j$ and let $S = K\setminus \{b\}$.  Since $K$ is a circuit, we have $\vee K = \vee S = \vee T$.  On the other hand, $\vee S \leq z_{j-1}$ since, by the choice of~$j$, all elements of $S$ are less than or equal to  $z_{j-1}$.  But $\vee T\not\leq z_{j-1}$ since $b\in T$ and $b\not\leq z_{j-1}$.  This is a contradiction.
\end{proof}

\begin{cor}\label{isfNBCCor}
Let $G$ be a labeled $n$-multigraph.  Then $\isf_{m}(G) \leq \nbc_m(L(G))$ for all $m$.  Moreover, $\isf_{m}(G) =\nbc_m(L(G))$ for all $m$ if and only if $G$ is perfectly labeled.
\end{cor}
\begin{proof}
The inequality follows from Proposition~\ref{transNBC}, together with the earlier observation that the numbers of ISFs with $m$ edges equals the number of atomic transversals of size~$m$.  The second assertion follows from Rota's formula~\eqref{Rota} together with Theorem~\ref{ISFandChi}.
\end{proof}

Next we discuss consequences for the topology of the complement of $\AA_G$ and its real analogue.  Recall that the \emph{$i^{th}$ Betti number} $\beta_i(X)$ of a topological space $X$ is the dimension of the $i^{th}$ (unreduced) homology group of $X$ with coefficients in~$\Rr$.   Let $\AA$ be an arrangement in $\mathbb{C}^n$ and let $X(\AA)=\mathbb{C}^n\setminus \cup\AA$.  A well-known result of Orlik and Solomon~\cite[Theorem 5.2]{os:ctch} states that $\beta_m(X(\AA)) = \nbc_{n-m}(L(\AA))$.   An immediate consequence of the previous result  is the following.

\begin{cor} \label{betti}
Let $G$ be a labeled $n$-multigraph.  Then
\[
\isf_m(G) \leq \beta_{n-m}(X(\AA(G)))
\]
for all $m$ with equality for all $m$ if and only if $G$ is perfectly labeled.
\end{cor}

We now consider the case that $G$ is an \emph{$\Rr$-labeled multigraph}, i.e., all labels are nonzero real numbers.  Then formula~\eqref{define-arrangement} may be viewed as defining an arrangement~$\AA_\Rr(G)$ in $\Rr^n$, with the same intersection lattice as the complex arrangement~$\AA(G)$.  As before, we define $X(\AA_\Rr)=\Rr^n\setminus \cup\AA$.  Zaslavsky~\cite{z:fua} showed that for every real hyperplane arrangement $\AA$,  the number of regions $r(\AA)$ of $X(\AA)$ is given by $\nbc(L(\AA))$.  Observe that the proofs of Theorem~\ref{ISFandChi} and Proposition~\ref{transNBC} go through without change upon replacing $\AA(G)$ by $\AA_\Rr(G)$, as does Corollary~\ref{isfNBCCor}, with the following consequence.

\begin{cor}\label{region}
Let $G$ be an $\Rr$-labeled $n$-multigraph.  Then
$$
\isf(G)\leq r(\AA_\Rr(G))
$$
with equality if and only if $G$ is perfectly labeled.
\end{cor}

We now mention a generalization of Theorem~\ref{HS2} to signed graphs.  Let $G$ be a labeled $n$-multigraph with at most two edges between any two nonzero vertices.  If every edge between nonzero vertices of $G$ is labeled by $1$ or $-1$ we say that $G$ is a \emph{signed graph}.  This is essentially equivalent to the \emph{signed graphs} studied by Zaslavsky in~\cite{Z}.  Our signed graphs have the additional vertex~0, and an edge of the form~$0i$ in our setting  corresponds to a half-edge at~$i$ in~\cite{Z}.

A \emph{coloring} of a signed graph $G$ is a function 
$$
c:V(G)\setminus\{0\}\to[-s,s] = \{-s,-s+1,\dots, 0, \dots, s-1,s\},
$$ 
where $s$ is some nonnegative integer.  A coloring is \emph{proper} provided that
\ben
\item[(i)] for all $i,j\neq 0$, if $ij^\epsilon\in E(G)$, then $c(i)\neq \epsilon c(j)$; and
\item[(ii)] if $0i\in E(G)$, then $c(i) \neq 0$. 
\een
These definitions correspond to those in~\cite{Z}, and our real hyperplane arrangement $\AA_\Rr(G)$ coincides with Zaslavsky's $H[G]$.

As for an ordinary graph, the \emph{chromatic function} $P(G,t)$ of a signed graph $G$ is the number of proper colorings of $G$ with $t=2s+1$ colors.  It follows from~\cite[Theorem~2.2]{Z} that if $G$ is a signed graph with vertex set $\{0,1\dots,n\}$, then $P(G,t)$ is a polynomial, specifically,
$$
P(G,t) = t^{n-\rho(L)}\chi(L(\AA_\Rr(G),t)).
$$
Using this fact and Theorem~\ref{ISFandChi}, we have the following generalization of Theorem~\ref{HS2}.

\begin{thm}
Let $G$ be a signed graph with vertex set $\{0,1\dots,n\}$.  Then
$$
\ISF(G,t) = (-1)^n P(G,-t)
$$
if and only if $G$ is perfectly labeled.
\end{thm}

We finish this section with a result on supersolvable arrangements. For more about such arrangements,  
see~\cite[\S4.3]{s:aiha}.  In~\cite[Proposition 22]{hs:fcp}, it was shown that if $L$ is a geometric lattice with a  $\hat{0}$-$\hat{1}$ saturated chain that induces a partition of the atom set $(A_1,A_2,\dots, A_n)$, and the characteristic polynomial of $L$ factors as $\chi(L,t) = \prod_{i=1}^n (t-|A_i|)$, then $L$ is supersolvable. The converse of this statement was shown earlier by Stanley~\cite[Theorem 4.1]{sta:sl}.

\begin{prop} \label{supersolvable}
Let $G$ be a labeled $n$-multigraph.  If $G$ is perfectly labeled, then the lattice $L(G)=L(\AA(G))$ is supersolvable.
\end{prop}
\begin{proof}

Consider the multichain $\hat{0}=V_0\leq V_1\leq\cdots \leq V_n=\hat{1}$ defined in the proof of Theorem~\ref{ISFandChi}.   Let $C'$ be the chain obtained from $C$ by removing any  repeated elements in $C$. We claim that $C'$ is saturated.  Indeed, suppose $V_k<V_{k+1}$.  For $\ell\ge1$, let $M_\ell$ be the matrix whose rows are  normal vectors to the hyperplanes defining $V_\ell$, where for a hyperplane of the form $x_i=\alpha x_j$, $i<j$, we take the normal with a one in coordinate $i$ and similarly for $x_i=0$.  We denote these row vectors by $r(x_i=\alpha x_j)$ and $r(x_i=0)$, respectively.    Since $V_\ell$ is the nullspace of $M_\ell$, it suffices to show that
\beq
\label{rk}
\rk M_{k+1} = 1+\rk M_k
\eeq
where $\rk$ denotes matrix rank.  

Equation~\ree{rk}  is clearly true if $M_{k+1}$ has only one more row than $M_k$.  So suppose at least two hyperplanes were added; then one of them must be of the form $x_j = \beta x_{k+1}$.  Consider such a hyperplane where $j$ is maximum.  It is obvious that $\rk M' = 1+ \rk M_k$ if $M'$ is the matrix obtained by adding $r(x_j = \beta x_{k+1})$ to $M_k$.  So we will be done if we can show that the other rows of $M_{k+1}$ involving $x_{k+1}$ are linear combinations of the rows of $M'$.  But this follows from the fact that $G$ is perfectly labeled.    For example, consider a row corresponding to a hyperplane
$x_i = \alpha x_{k+1}$ with $i<j$.  By condition (1) of Definition~\ref{perfect-labeling}, the matrix $M_k$ contains the row $r(x_i = \alpha/\beta x_j)$, and
$$
r(x_i = \alpha x_{k+1}) = (\alpha/\beta) r(x_j=\beta x_{k+1}) + r(x_i = (\alpha/\beta) x_j).
$$

Using Theorems~\ref{HS1} and~\ref{ISFandChi} we see that $\chi(L(G),t)$ factors with nonpositive integer roots.  By the previous paragraph, this factorization is induced by the saturated chain $C'$.  It follows from~\cite[Proposition~22]{hs:fcp} that $L(G)$ is supersolvable.
\end{proof}

If $G$ is a simple graph, 0 is an isolated vertex and all edges are labeled by 1, then a perfect labeling of~$G$ is just a PEO and the converse of Proposition~\ref{supersolvable} holds; see, e.g.,~\cite[Corollary 4.10]{s:aiha}.  However, the converse is false for multigraphs.  For example, the multigraph shown in Figure~\ref{supersolvable-no-perfect} has no perfect labeling, since condition~(2) of Definition~\ref{perfect-labeling} must fail, but on the other hand $L(G)$ is easily seen to be supersolvable.

\begin{figure}
\begin{center}
\begin{tikzpicture}
\foreach \v in {(0,0),(2,0),(4,0)} \fill \v circle(.1);
\node at (0,-.4) {0};
\node at (2,-.4) {1};
\node at (4,-.4) {2};
\draw (2,0) to[bend left] (4,0) to[bend left] (2,0);
\draw(3, .6) node {$\alpha$};
\draw(3, -.6) node {$\beta$};
\end{tikzpicture}
\end{center}
\caption{A supersolvable multigraph with no perfect labeling.\label{supersolvable-no-perfect}}
\end{figure}
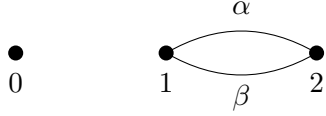


\section{Forests and pattern avoidance} \label{inv}

In this section, we study \emph{tight forests}, another class of labeled forests characterized by avoiding certain permutation patterns.  If $G$ is triangle-free, then every tight spanning forest is an NBC~set, and the converse is true if the vertex labeling is a \emph{quasi-perfect ordering} (QPO).  These orderings exist only for certain bipartite graphs; when they do, they give a combinatorial interpretation of the chromatic polynomial as a generating function for tight spanning forests.

Two sequences of distinct positive integers $\pi =\pi_1\pi_2\cdots \pi_k$ and $\sigma= \sigma_1\sigma_2\cdots \sigma_k$ are \emph{order-isomorphic}, written $\pi\sim\sigma$, provided that $\pi_i<\pi_j$ if and only if $\sigma_i<\sigma_j$ for all $i$ and $j$.  For example, $2341\sim6892$.  If $\pi$ and $\sigma$ are two sequences of  distinct positive integers we say \emph{$\si$ contains $\pi$ as a pattern} if $\si$ has a  subsequence that is order-isomorphic to $\pi$.  For example, $\si = 6892$ contains $231$ as a pattern because of the subsequence $692\sim231$.   We say that \emph{$\si$ avoids $\pi$} if $\si$ does not contain $\pi$ as a pattern.  For example, 6892 avoids 321.  More generally, given a set of sequences  $\Pi$ we say that \emph{$\si$ avoids $\Pi$} if $\si$ avoids every 
$\pi\in \Pi$.

Now let $T$ be a tree with vertices labeled by distinct positive integers.  As before, we regard $T$ as rooted at its smallest vertex $r$.  We say that $T$ \emph{avoids} $\pi$ if every path from $r$ to a leaf of~$T$ avoids $\pi$.  For example, if $T$ is the right-hand tree in Figure~\ref{T}, then the paths from the root are $235, 239, 26, 278,  274$, so $T$ avoids 321 but not 21 because $74\sim21$.  Indeed, a tree is increasing precisely if it avoids 21.  We say  that a labeled forest avoids $\pi$ if every component tree in it does so.  We similarly extend all definitions of pattern avoidance in sequences to forests.

For the remainder of the section, we will specifically consider labeled forests avoiding the set $\Pi=\{231,312,321\}$.  To simplify discussion, we will call a copy of one of these permutations in a tree a \emph{bad triple}.  Note that a bad triple can never contain the root of a tree, since none of the forbidden patterns starts with the digit~1.  As shown in \cite[Prop.~6.2]{DDJSS}, the permutations avoiding these patterns are precisely the involutions such that all two-cycles are of the form $(i,i+1)$, and in particular are counted by the Fibonacci numbers.

\begin{defn}\label{def:tight}
A permutation is \emph{tight} if it avoids the patterns $\Pi=\{231, 312, 321\}$.  Equivalently, all the 1's in its permutation matrix lie on or adjacent to the main diagonal.  A sequence of distinct integers is tight if it is order-isomorphic to a tight permutation.  Finally, if $F$ is a labeled forest, with every component tree rooted at its smallest vertex, then we saw that $F$ is tight if every path starting at a root is a tight sequence.
\end{defn}

It is worth noting that every tree with two or fewer edges is tight.
As before, we use the notation
\begin{equation} \label{TF-notation}
\begin{aligned}
\cTF(G)&=\text{set of tight spanning forests of $G$}, & \tf(G) &= |\cTF(G)|,\\
\cTF_m(G)&=\text{set of tight spanning forests with $m$ edges}, & \tf_m(G) &= |\cTF_m(G)|,\\
\TF(G) &= \TF(G,t)=\sum_{m\ge0} \tf_m(G) t^{n-m}.
\end{aligned}
\end{equation}

\begin{lem}\label{subTreeAvoidLem}
Every subforest of a tight labeled forest is tight.
\end{lem}

\begin{proof}
It suffices to prove the corresponding statement for labeled trees.  Accordingly, let $T$ be a labeled tree with root $r$, and let $T'$ be a subtree of $T$.  Let $r'$ be the root of $T'$, and let $s$ be the vertex of $T'$ which is closest to $r$ in $T$, so that $P:r,\dots,s,\dots,r'$ is a path in $T$.  Since $r'$ has the smallest label of any vertex in $T'$, it follows that $r'$ and $s$ are either identical or adjacent; otherwise the path $s,\dots,r'$   would contain the pattern $321$ or $231$, which would contradict the fact that $P$ is tight.

Now let $P'$ be a path in $T'$ starting at $r'$.  By the previous paragraph, $P'$ is either a subpath of some path from $r$ in $T$, or else it has the form $r',s,P''$, where $s$ is the parent of $r'$ in $T$ and $s,P''$ is a subpath of a path from $r$ in $T$.  Then $s,P''$ is tight because $T$ is tight, and $r',s,P''$ is tight as well, again by Definition~\ref{def:tight}, since the label of $r'$ is smaller than that of any other vertex in this path.
\end{proof}

In Theorem~\ref{isf-and-nbc}, we showed that every increasing forest is an NBC~set.  We wish to extend this result to the setting of tight forests.  It is not the case that every tight forest is an NBC~set: for example, in the cycle with vertex set $[3]$ the three-vertex path labeled $1, 3, 2$  is a tight forest but is itself a broken circuit.  On the other hand, the next result shows that broken 3-cycles are the only obstruction to extending Theorem~\ref{isf-and-nbc}.

\begin{prop}\label{avoidIsNBCProp}
Let $G$ be a graph with no  3-cycles. Then $\cTF_k(G) \subseteq \cNBC_k(G)$ for all $k$.
\end{prop}

\begin{proof}
By Lemma~\ref{subTreeAvoidLem}, it is equivalent to show that any broken circuit $B$  of length 4 or more, regarded as a labeled tree contains a bad triple.  Now $B$  is a path of the form $a,x,y,\dots,b$ where $a=\min(B)$ and $b<x$.  Then either $b<x<y$, $b<y<x$, or $y<b<x$.  In each case $x,y,b$ is a bad triple, being a copy of  $231$, $321$, or $312$ respectively.
\end{proof}

\begin{prop}\label{3cycleNBCneq}
If $G$ has a 3-cycle, then $\cTF_2(G)\subsetneq\cNBC_2(G)$ and so $\TF(G,t)\neq (-1)^{|V(G)|}P(G,-t)$.  
\end{prop}
\begin{proof}
Let $G$ have a $3$-cycle $C$.  So $C$ is order isomorphic to the cycle on $[3]$, and the example before the previous proposition indicates how to find an element of $\cTF_2(G)$ not in $\cNBC_2(G)$.
\end{proof}

We now define a kind of vertex ordering that is related to tight forests as PEOs are to ISFs.

\begin{defn} \label{def:QPO}
Let $G$ be a graph.  A \emph{candidate path} in $G$ is a path of the form
\begin{equation} \label{candidate-path}
a,\ c,\ b,\ v_1,\ \dots,\ v_m=d
\end{equation}
such that $a<b<c$; $m\geq 1$; and $v_m$ is the only $v_i$ smaller than $c$. The total ordering of $V(G)$ is called a \emph{quasi-perfect ordering} (QPO) if all candidate paths satisfy the following condition:
\begin{equation} \label{QPO-fail} 
\text{either } ad\in E(G), \text{ or else } d<b \text{ and } cd\in E(G).
\end{equation}
\end{defn}

This definition seems obscure, but in fact QPOs are an extension of PEOs in the following sense.  If the possibility $m=0$ were allowed, so that $d=b$, then~\ree{QPO-fail} would reduce to the requirement that $ab\in E(G)$, just as in a PEO.  The existence of a QPO has strong structural consequences, as we now explain.

\begin{prop}\label{chord}
Let $G$ be a graph with a QPO.
\begin{enumerate}
\item[(1)] Every cycle of length at least $5$ has a chord.
\item[(2)] If $G$  has no $3$-cycles, then it is bipartite.
\end{enumerate}
\end{prop}

\begin{proof}
For (1), we prove the contrapositive.  Suppose that $C\subseteq G$ is a chordless cycle of length at least 5.  Let $P$ be the three-edge subpath $a,c,b,d$, where $c=\max(C)$ and $b>a$.  Neither $ad$ nor $cd$ are edges, so $P$ is a candidate path that fails~\eqref{QPO-fail}.

For (2), if $G$  is not bipartite, then it has an odd cycle  $C$.   If $C$ has minimum length, then it cannot be a cycle of length at least $5$ by part (a).  So $C$ must be a triangle, which is a contradiction.
\end{proof}

Before continuing with the general development, we will look at a few examples.

\begin{exa}
We note that not every graph with a QPO need be bipartite.
For example, the labeling of the non-bipartite graph shown in Figure~\ref{QPO:nonbipartite} is a QPO.  In particular, the only candidate path is $a,c,b,d=1,5,4,3$, and $ad$ is an edge.
\end{exa}
\begin{figure}
\begin{center}
\begin{tikzpicture}
\foreach \v in {(0,0),(2,0),(0,2),(2,2),(4,1)} \fill \v circle (.1);
\draw (2,2) -- (0,2) -- (0,0) -- (2,0) -- (2,2) -- (4,1) -- (2,0);
\node at (-.4,0) {5};
\node at (-.4,2) {4};
\node at (2.6,0) {1};
\node at (2.6,2) {3};
\node at (4.4,1) {2};
\end{tikzpicture}
\end{center}
\caption{A non-bipartite graph with a QPO\label{QPO:nonbipartite}}
\end{figure}
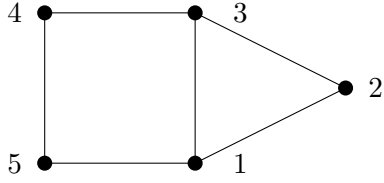

\begin{exa}
For every positive integer $m$, the complete bipartite graphs $K_{m,1}$, $K_{m,2}$ and $K_{m,3}$ all admit QPOs.  For $K_{m,1}$, every ordering is a QPO.  For $K_{m,2}$ one can label the partite sets as $\{1,m+2\}$ and $\{2,\dots,m+1\}$; this gives a QPO, since every candidate path comes from removing an edge from a $4$-cycle.  For $K_{m,3}$, label the partite sets $X=\{1,2,N\}$ and $Y=\{3,4,\dots,N-1\}$, where $N=m+3$.  We claim that this labeling is a QPO. To show this, let $P$ be a candidate path, labeled as in \eqref{candidate-path}. If $N\not\in V(P)$, then we are done by the $K_{m,2}$ case.  And if $P$ has odd length then $av\in E(G)$.  Now consider $N\in V(P)$ and $P$ of even length.  So $cv\in E(G)$.  If $a\in X$ then we must have $a=1$, $b=2$, and $v=N$ which is a contradiction.  If $a\in Y$ then this forces $c=N$ and $v=1$ or $2$.  So $v<b$ and~\ree{QPO-fail} is satisfied.
\end{exa}

\begin{exa}
Consider the complete bipartite graph $G=K_{m,n}$, where $m,n\geq 4$.  Note that every cycle in~$G$ of length $\geq5$ has a chord.  But we claim that no ordering of $V(G)$ is a QPO.  Suppose, towards a contradiction,  that~$G$ has a QPO.  Let $N=m+n$ and let $X,Y$ be the partite sets of~$G$.

First, suppose that $N-1,N$ belong to the same partite set, say~$X$.  Let $a<b<d$ be vertices in~$Y$.  Then
$a,\ c=N-1,\ b,\ N,\ d$
is a candidate path with $ad\not\in E(G)$ and $d>b$, so condition~\eqref{QPO-fail} fails.

Second, suppose that $N-1,N$ belong to different partite sets.  Let $X$ be the partite set containing~$N-2$ and let $a,b,d\in Y$ with $a<b<d<N-2$; note that three such vertices must exist because $|Y|\geq4$.  Let~$x$ be whichever of~$N-1,N$ belongs to $X$.  Then
$a,\ c=N-2,\ b,\ x,\ d$
is a candidate path leading to the same failure of~\ree{QPO-fail} as in the previous case.
\end{exa}

For graphs with no 3-cycles, having a QPO implies that the tight forests are precisely the NBC~sets.

\begin{prop} \label{NBCSetAvoidProp}
Let $G$ be a labeled graph with no 3-cycles.  If the vertex labeling is a QPO, then $\cTF(G)=\cNBC(G)$.
\end{prop}
\begin{proof}
By Proposition~\ref{avoidIsNBCProp}, we only need to prove $\cNBC(G)\sbe\cTF(G)$.  So let $F$ be an NBC set in $G$.  Then $F$ is a forest, since if $F$ contained a cycle then it would contain a broken circuit.  Let $Q$ be a candidate path in $F$ labeled as in~\ree{candidate-path}.  Then $ad\not\in E(G)$ since if the edge were present in $G$, then $Q$ would be a broken circuit in $F$.  It follows that we must have both $d<b$ and $cd\in E(G)$.  But this rules out any candidate path with four vertices since such a path together with $cd$ would contain a $3$-cycle.

Now let $T$ be one of the component trees of $F$.  To finish the proof, we will show that any path $P$ starting at the root of $T$ is tight by induction on the number of vertices in $P$.  This is clear if $|V(P)|\le2$ since the patterns to be avoided all have $3$ elements.  So assume the result for paths with $k$ vertices and consider a path $P=w_1,w_2,\dots,w_{k+1}$ from the root of $T$.  Let $\pi=\pi_1\pi_2\dots\pi_{k+1}$ be the standardization of $P$, that is, the unique permutation of $[k+1]$ with $\pi\sim P$.  Since $P$ starts at the minimum vertex of $T$ we have  $\pi_1=1$.  So it suffices to show that $\pi$ is a tight involution.  Our main tool will be the characterization of tight involutions in terms of fixed points and $2$-cycles
in Definition~\ref{def:tight}.  By induction $\pi^-:=\pi_1\pi_2\dots\pi_k$ is a tight sequence.  So there are only three possibilities for the position of $k+1$ in $\pi$.

Case 1: $\pi_{k+1}=k+1$.  In this case $\pi^-$ is a tight involution on $[k]$.  So the concatenation $\pi=\pi^-,k+1$ just adds a fixed point at $k+1$ and is also tight.

Case 2: $\pi_k=k+1$.  Since $\pi^-$ is tight, there are three possibilities for the position of $k$ in $\pi$.  If $\pi_{k+1}=k$ then we are done by a similar argument as in Case 1, only adding the $2$-cycle $(k,k+1)$.  If $\pi_{k-1}=k$ then  
$\pi_{k+1},k+1,k,\pi_{k-2}$ corresponds to a candidate path in~$T$ with four vertices, a contradiction as shown in the first paragraph of the proof.  Finally, suppose $\pi_{k-2}=k$. 
First, note that $k>3$ otherwise $\pi_1=3\neq 1$.  So $\pi_{k-3}$ exists. 
Also, since $\pi^-$ is tight and both $\pi_{k-2}=k$ and $\pi_k=k+1$, we must have $\pi_{k-1}=k-1$.  Thus the sequence 
$\pi_{k-3},k,k-1,k+1,\pi_{k+1}$ corresponds to a candidate path
$w_{k-3}, w_{k-2}, w_{k-1}, w_k, w_{k+1}$.  
Appealing to the first paragraph again, $w_{k+1}<w_{k-1}$ and 
$w_{k-2} w_{k+1}\in E(G)$.  But then the path $w_{k-2},w_{k-1},w_k,w_{k+1}$ is a broken circuit in $T$, another contradiction.

Case 3: $\pi_{k-1}=k+1$.  Since $\pi^-$ is tight, there are only two possible positions for $k$ in $\pi$.  If $\pi_{k+1}=k$ then, again by the tightness of $\pi^-$, we must have $\pi_k=k-1$.  This results in the contradiction that $\pi_{k-2},k+1,k-1,k$ corresponds to a candidate path with four vertices.  Finally, if $\pi_k=k$ then we have another four-vertex candidate path corresponding to 
$\pi_{k-2},k+1,k,\pi_{k+1}$.  This final contradiction completes the proof.
\end{proof}

\begin{thm}\label{orderingIffNBC}
Suppose that $G$ has vertex set $[n]$ and has no 3-cycles.   Then the following are equivalent:
\begin{enumerate}
\item[(1)]\label{hasQPO} The labeling of vertices is a QPO.
\item[(2)]\label{NBCisTF}   $\cNBC(G) =  \cTF(G)$.
\item[(3)]\label{TFisPGt} $\TF(G,t)  = (-1)^n P(G,-t)$.
\end{enumerate}
\end{thm}
\begin{proof}
The equivalence of~(2) and~(3) follows from Whitney's formula (Theorem~\ref{Whitney}), and the implication $(1)\Rightarrow(2)$ is just Proposition~\ref{NBCSetAvoidProp}, so it remains to prove the converse.  Accordingly, suppose that the labeling of vertices is not a QPO.  Let $Q$ be a candidate path, labeled as in~\eqref{candidate-path}, that fails~\eqref{QPO-fail} and for which $m$ is as small as possible.   We will show that $Q$ contains a bad pattern and is an NBC~set which contradicts condition~(2).

Either $a$ or $d$ is the smallest vertex of~$Q$, so the smaller of the two must be the root.  If $a<d$ then $c,b,d$ is either a 321-pattern or a 312-pattern according as $d<b$ or $d>b$.  On the other hand, if $a>d$ then $b,c,a$ is a 231-pattern.

Now we show that $Q$ is an NBC~set---that is, if $e$ is an edge outside $Q$ whose endpoints belong to $Q$, then $e$ is not the smallest edge of the resulting cycle.  We consider the possibilities for $e$ individually.

 First, consider possible edges containing vertex $a$.  We know $e=ab$ is not an edge of $G$ since $G$ has no 3-cycles, and $e=ad$ is not an edge of $G$ since $Q$ fails the QPO condition.  This leaves $e=av_i$ for some $i \in [m-1]$, but then $e$ is lexicographically greater than $ac$.

Second, we look at edges containing $c$.  If $e=cv_i$ for some $i\in [m-1]$, then $e$ is lexicographically greater than $bc$.  If $e=cd$ and $d>b$ then $e$ is still lexicographically greater than $bc$.  The only other possibility is when $d<b$.  But since $Q$ does not satisfy the QPO condition, this forces $cd$ not to be an edge of $G$.

Finally, suppose $e=v_iv_j$ for some $i,j$ with $j\geq i+3$ where we let $v_0=b$.  In this case, 
\[a,\ c,\ b,\ v_1,\ \dots,\ v_{i-1},\ v_i,\ v_j,\ v_{j+1},\ \dots,\ v_m=d\]
is a shorter candidate path that fails~\eqref{QPO-fail}, which contradicts the choice of~$Q$.
\end{proof}

By Theorem~\ref{HS1}, the generating function for increasing forests in any graph has only nonpositive integer roots, regardless of the ordering of the vertices.  This is not in general the case for the corresponding generating function for tight forests.

\begin{prop}
Let $G$ be a graph with $n$ vertices and $e$ edges.  Then $G$ has a vertex ordering such that every root of  $\TF(G,t)$ is an integer  if and only if it is a forest.
\end{prop}
\begin{proof}

\noindent ($\Leftarrow$)  If $G$ is a forest, then any increasing labeling gives both a PEO and a QPO.  Hence $\TF(G,t)=\ISF(G,t)$, which has only integral roots by Theorem~\ref{HS1}.

\noindent ($\Rightarrow$)  Suppose that $\TF(G,t)=(t+a_1)\cdots(t+a_n)$, where the $a_n$ are integers.  Since all coefficients of $\TF(G,t)$ are nonnegative,  all its roots are nonpositive  and  so  all the $a_i$ are nonnegative.  As observed earlier, every forest with two or fewer edges is tight so we have, for $q=|E(G)|$
\begin{align*} 
\TF(G,t) &= t^n+q t^{n-1}+\binom{q}{2}t^{n-2}+\cdots\\
&= t^n + \left(\sum_i a_i\right)t^{n-1} + \left(\sum_{i<j} a_ia_j\right)t^{n-2}+\cdots.
\end{align*}
Note that
\[\binom{q}{2} = \binom{a_1+\cdots+a_n}{2}
 = \frac12\left(\sum_i a_i^2+2\sum_{i<j} a_ia_j-\sum_i a_i\right)
 = \frac12\left(\sum_i (a_i^2-a_i)\right)+\sum_{i<j} a_ia_j.\]
 So for the second equality in the first sequence of equations to hold, we must have $a_i\in\{0,1\}$ for all $i$.  Since $q=\sum a_i$ we obtain $\TF(G,t)=t^{n-q}(t+1)^q$.   But this implies that every subset of $E(G)$ is a (tight) forest, so $G$ is a forest.
\end{proof}

We conclude with two problems for further study suggested by the previous proposition.  For any set of patterns $\Pi$ and any graph~$G$, let $\mathcal{F}^\Pi_m(G)$ be the set of $m$-edge spanning forests of $G$ avoiding all patterns in $\Pi$, and let $F^\Pi(G,t)=\sum_m  |\mathcal{F}^\Pi_m(G)| t^{n-m}$.

\begin{question}
For which sets of patterns $\Pi$ and graphs $G$ does $F^\Pi(G,t)$ have integer roots?  Can there be roots other than 0 and $-1$?
\end{question}

For $\Pi= \{21\}$, the generating function $F^\Pi(G,t)=\ISF(G,t)$ has integer roots by Theorem~\ref{HS1}, so its coefficient sequence is log-concave, hence unimodal.  For $\Pi=\{231,312,321\}$, if the vertex labeling is a QPO, then by Theorem~\ref{orderingIffNBC} we have $F^\Pi(G,t)=\TF(G,t)=(-1)^n P(G,-t)$, and the coefficient sequence is log-concave by a celebrated recent result of Huh~\cite{h:mnphcp}.

\begin{question}
For which sets of patterns $\Pi$ and graphs $G$ is the coefficient sequence of $F^\Pi(G,t)$ unimodal or log-concave?
\end{question}

{\em Acknowledgement.} We thank Art Duval for stimulating conversations.

\bibliographystyle{amsalpha}
\newcommand{\etalchar}[1]{$^{#1}$}
\providecommand{\MR}[1]{}
\providecommand{\bysame}{\leavevmode\hbox to3em{\hrulefill}\thinspace}
\providecommand{\MR}{\relax\ifhmode\unskip\space\fi MR }
\providecommand{\MRhref}[2]{%
  \href{http://www.ams.org/mathscinet-getitem?mr=#1}{#2}
}
\providecommand{\href}[2]{#2}

\end{document}